\documentclass[12pt]{amsart}
\usepackage[left=3cm,top=3.0cm,right=3cm,bottom=2.6cm]{geometry}

\usepackage[ansinew]{inputenc}
\usepackage{amscd,amsfonts,amsmath,amssymb,amsthm,enumerate,epsfig,float,graphics,graphicx,pst-all,yhmath}
\newcommand{\inte}{\operatorname*{int}}
\newcommand{\aff}{\operatorname*{aff}}
\newcommand{\bd}{\operatorname*{bd}}

 \newcommand{\conv}{\operatorname*{conv}}

\newcommand{\Rd}{\mathbb{R}^2}
\newcommand{\Rt}{\mathbb{R}^3}
\newcommand{\Rn}{\mathbb{R}^{n}}

\newtheorem{lemma}{Lemma}

\newtheorem{theorem}{Theorem}
\newtheorem{corollary}{Corollary}

\newtheorem{remark}{Remark}
\newtheorem{definition}{Definition}

\newtheorem{claim}{Claim}

\newcommand{\fin}{\hfill $\Box$}

\title {Convex bodies with centrally symmetric sections}
\author{E. Morales-Amaya}
\address{Facultad de Matem\'aticas-Acapulco,
Universidad Aut\'onoma de Guerrero, M\'exico}
\email{13529@uagro.mx}

\begin{document}

 \begin{abstract}  
                  Let $K\subset \mathbb{R}^n$ be a convex body, $n\geq 3$. We say that $K$ 
                  satisfies the \textit{Barker-Larman condition} if there exists a ball 
                  $B\subset \text{int} K$ such that for every suppor\-t hyperplane $\Pi$ of $B$, 
                  the section $\Pi \cap K$ is a centrally symmetric set. 
                  
                   In \cite{Barker}, it was conjectured that the Barker-Larman condition characterizes the 
                   e\-llip\-soid. In this work, we prove a special case of such a conjecture; in particular, we 
                   assume that the convex body $K$ is centrally symmetric. Our main result is
                  the following: Let $K$ be a centrally symmetric and strictly convex body,    
                   with center at $O$, and let $B\subset \inte K$ be a ball not containing $O$: If $K$  
                   satisfies the Barker-Larman condition with respect to $B$ and $B$ is \textit{suitable} 
                   for $K$ (intuitively, $B$ is suitable for $K$ if $\bd B$ \textit{is not very close to} $\bd K$, 
                   see the definition in the Introduction), then $K$ is an ellipsoid. 
\end{abstract}
\maketitle

\section{Introduction}

                  To present our first result, we need the following definitions. 
                  Let $K\subset \mathbb{R}^n$ be a centrally symmetric convex body with center at 
                  $O$ and let $B \subset \inte K$ be a ball such that $O\notin B$. We take a system 
                  of coordinates for $\Rn$ such that $O$ is the origin. On the one hand, the ball $B$ is 
                  said to be \textit{suitable} for $K$ if, for every support hyperplane 
                  $\Gamma$ of   $G:=\conv[B\cup (-B)]$, the relation
\begin{eqnarray}\label{jazzu}
                  B\subset \conv (K_{\Gamma} \cup K_{-\Gamma} )  
\end{eqnarray} 
                  holds, where $K_{\Gamma}:=\Gamma \cap K $ and 
                  $K_{-\Gamma}:=-\Gamma \cap K $. On the other hand, we say that $K$ satisfies 
                  the \textit{Barker-Larman Condition} if there exists a ball $B\subset \text{int} K$ such 
                  that, for every support hyperplane $\Pi$ of $B$, the section $\Pi \cap K$ is 
                  centrally symmetric.    
 
                   We are now ready to present our main result.
                   
\begin{theorem}\label{mozart}
                  Let $K\subset \mathbb{R}^n$, $n\geq3 $, be a centrally symmetric and strictly 
                  convex body with center at $O$ and let $B\subset \inte K$ be a ball with $O\notin B$. 
                  Suppose that $K$ and $B$ are such that: 1) $K$ satisfies the Barker-Larman condition 
                  with respect to $B$, and 2) $B$ is suitable for $K$. Then $K$ is an ellipsoid.
\end{theorem}
                  Notice that if $K$ is a centrally symmetric convex body with center at $O$, and it 
                  satisfies the Barker-Larman condition with respect to the ball $B$, $O\notin B$, 
                  which is suitable for $K$, then the set $\conv (K_{\Gamma} \cup K_{-\Gamma})$ is 
                  a cylinder and  the condition (\ref{jazzu}) means that $B$ is contained in it for all 
                  support planes $\Gamma$ of $G$.

                     An ellipsoid $E\subset \Rn$ has the property that all its sections with $k$-planes, 
                     $2\leq k\leq n-1$, are ellipsoids. Therefore, its sections have symmetries depending on 
                     the context in which the ellipsoid is considered; that is, for example, for $k=n-1$, in 
                     the Euclidean case, the sections of $E$ with hyperplanes have centers of symmetry, 
                     axes of symmetry, and $(n-2)$-planes of symmetry (the formal definitions of these 
                     concepts will be presented in the next section). If we consider the affine context, the 
                     sections of $E$ with hyperplanes have centers of symmetry, affine axes of symmetry, 
                     and $(n-2)$-planes of affine symmetry (the formal definitions of these concepts will be 
                     presented in the next section). Finally, in the projective context, the sections of $E$ 
                     with hyperplanes, there are poles and polars (although these concepts will not be 
                     considered in this work, the interested reader may consult such notions and related 
                     results in \cite{Kelly},  \cite{MM1}, \cite{larmanmorales}). It is natural to ask, in the 
                     Euclidean context, 
                     
                     \textit{if we have a convex body $K\subset \Rn$, $n\geq 3$, such that all its 
                     hypersections have centers of symmetry; must $K$ be an ellipsoid?} 
                     
                     Naturally, the corresponding questions arise for the case of axes of symmetry and 
                     $(n-2)$-planes of symmetry (see \cite{mo}, \cite{gmm}, \cite{alfonseca}), and also 
                     the corresponding questions when mo\-ving to 
                     affine and projective contexts. Since we have a positive answer to the previous 
                     question (the corresponding theorem was first proved by Brunn \cite{Brunn} in 1889 under 
                     the hypothesis of regularity and was proved in general by Burton \cite{Burton}), 
                     we can try to 
                     reduce the number of sections of $K$ that have centers of 
                     symmetry and still guarantee that $K$ is an ellipsoid. This has been done in several 
                     ways. We will limit ourselves to presenting only two, which, in our opinion, are the 
                     significant ones (however, we also recommend seeing \cite{BG}, \cite{mm3}, 
                     \cite{LMM}). The first of this result is the following Olevjanischnikoff's 
                     Theorem \cite{Ol2} concerning homothetic ellipsoids.

\begin{theorem}[Olevjanischnikoff] \label{thmOle}
                     Let $K$ and $L$ be convex bodies in $\Rn$, with $L\subset \inte K$, such that every 
                     hypersection of $K$ tangent to $L$ is centrally symmetric and its center belongs to 
                     $\mbox{} L$, then $K$ and $L$ are homothetic and concentric ellipsoids.
\end{theorem}

                    The second one is the following result due to Rogers \cite{Rogers}:

\begin{theorem}[Rogers]
                    Let $K\subset\mathbb R^n$ be a convex body, $n\geq 3$, and 
                    $x\in\mathbb R^n$. If all the $2$-dimensional sections of $K$ through $x$ have a 
                    center of symmetry, then $K$ has a centre of symmetry.
\end{theorem}
                    Rogers also conjectured that if $x$ was not the center of $K$ then it is an ellipsoid. 
                    Aitchison, Petty, and Rogers first proved this conjecture \cite{false_centre} 
                    when $x \in \inte K$, and, in the general case, by Larman  \cite{larman}. Later, a simpler 
                    proof was given by Montejano and Morales-Amaya in \cite{Falso_MM}. Such a theorem 
                    is known as the False Centre Theorem.

                    Theorem~\ref{mozart} is a variant of the 
                    False Centre Theorem: When, instead of considering concurrent planes, we consider 
                    planes tangent to a given sphere. It can be considered as a progress on a problem 
                    due to Barker and Larman \cite{Barker}.  

                     \textbf{Structure of the paper.} The paper is organized as follows:
\begin{enumerate}
	\item [1.] Introduction (this section).
	\item [2.] Preliminaries.
        \item [3.] Proof of Theorem~\ref{mozart} for $n = 3$.
	\item [4.] Proof of Theorem \ref{mozart} in dimension $>3$. 
 \end{enumerate} 
                   Unfortunately, very few ideas from the False 
                   Centre Proof can be used in the proof of Theorem~\ref{mozart}. This is why it was 
                   necessary to implement new ideas to make some progress in proving the 
                   Barker-Larman Problem \cite{Barker}. It is worthwhile to give a brief description of 
                   the strategy for proving Theorem~\ref{mozart}.  In section 3, under the condition 
                   of Theorem~\ref{mozart} and assuming that the dimension is 3, 
                   we will prove: 1) the body $K$ has many planar shadow boundaries (see Lemmas 
                   \ref{star}, \ref{gato}, \ref{vivaldi}). 2) the body $K$ has many affine axes of symmetry. 
                   For a fixed affine axis of symmetry $L$, the planes containing $L$ define sections 
                   which are shadow boundaries of $K$, corresponding to directions parallel to some 
                   parallel support planes 
                   $\Pi_1, \Pi_2$
                   of $K$ at the points $L\cap \bd K$. This condition implies that all the sections of $K$, 
                   parallel to $\Pi_1$, are centrally symmetric with centers at $L$ and homothetic (see 
                   Lemma \ref{jamaica}). 
                   3) One of the affine axes of symmetry is, in fact, an affine axis of 
                   revolution (see Lemmas \ref{orchata}, \ref{limon}, \ref{tamarindo}). 4) $K$ has 
                   enough affine axes of symmetry to guarantee that the body $K$ is an ellipsoid.
                   
                   It is also worth noting that in \ref{preliminaries}. Preliminaries, we proved that 
                   the set of suitable balls contained in the ellipsoid $E_1\subset \Rn$, $n\geq 2$, 
                   and not containing $O$ is non-empty.
                  
\section{Preliminaries.}\label{preliminaries} 
                   Let $\mathbb{R}^{n}$ be the Euclidean space of dimension $n$ endowed with the 
                   usual inner pro\-duct 
                   $\langle \cdot, \cdot\rangle : \mathbb{R}^{n} \times \mathbb{R}^{n} \rightarrow 
                   \mathbb{R}$. We take an orthogonal system of coordinates $(x_1,...,x_{n} )$ for  
                   $\mathbb{R}^{n}$ and we denote by $O$ the origin. Let $\mathbb{S}^{n}=\{x\in 
                   \mathbb{R}^{n}: \|x\| = 1\}$ be the unit sphere. If $u\in \mathbb S^{n-1},$ then 
                   $u^\perp$ denotes the $(n-1)$-dimensional subspace orthogonal to $u$. Translations
                   of subspaces are called affine planes or $k$-planes.  Two planes are \emph{parallel} 
                   if one is contained in a translation of the other. Given a $k$-plane $H$, we denote 
                   by $H^\perp$ the $(n-k)$ subspace of $R^n$ orthogonal to $H$. If $C$ is a set 
                   $\text{aff }{C}$ denotes the smallest affine plane containing $C$. For the points 
                   $x,y \in \Rn$, we will denote by $L(x,y)$ the line defined by $x$ and $y$ and by 
                   $[x,y]$ the line segment contained in $L(x,y)$ with endpoints $x$ and $y$.   

                    A \textit{body} in $\Rn$ is a compact set which is equal to the closure of its 
                    nonempty interior. A \textit{convex body} is a body $K$ such that for every pair of 
                    points in $K$, the segment joining them is contained in $K$. A convex body is 
                    \textit{strictly convex} if its boundary does not contain a line segment. As usual 
                    $\inte K$, $\bd K$ will denote the \textit{interior} and the \textit{boundary} of the 
                    convex body $K$, respectively.

                    Given a $(n-1)$-plane $H\subset \Rn$ and a non-zero vector 
                    $u\in \Rn\backslash H$, the \textit{projection with respect to $u$ onto} $H$ is the 
                    linear transformation $\Phi_{H,u}:\Rn \to \Rn$ whose image $\Phi_{H,u}(\Rn)$ is 
                    $H$ and such that $\Phi_{H,u}^{-1}(p) =p+ \{\lambda u: \lambda \in \mathbb{R}\}$, 
                    for every $p\in H$. In the particular case when $u$ and $H$ are perpendicular this 
                    transformation will be denoted just by $\Phi_ u(\cdot)$. 

                    A set $C\subset \Rn$ is $O$-\emph {symmetric} if and only if $C=-C$.  Moreover, 
                    we say that $C$ is \emph{centrally symmetric} if there is a $O$-symmetric 
                    translated copy $C+w$ of $C$, i.e., $C=-C-2w$, that is, $C$ and $-C$ are translated. 
                    Notice that the converse is also true, i.e., if a set $W\subset \Rn$ and $-W$ are 
                    translated is because $W$ is centrally symmetric. 

                    For $n \geq 3$ we denote by $O(n)$ the \textit{orthogonal group}, i.e., the set of all 
                    the isometries of $\Rn$ that fix the origin. Let $K\subset  \Rn$ be a convex body, let 
                    $\Pi$ be an affine hyperplane, and $p$ be a point in $\Pi$. We denote by 
                    $O(\Pi,p,n-1)$ the set of all isometries of $\Pi$ that fix $p$. When it is clear which 
                    affine hyperplane $\Pi$ and point $p$ we are considering, we abuse the notation 
                    and write $O(n - 1)$ instead of $O(\Pi, p, n-1)$. The section $ \Pi\cap K$ is said to 
                    be \textit{symmetric} if there exists a non-trivial $\Omega \in O(n -1)$ such that 
\[
                    \Omega(\Pi \cap K) =  \Pi \cap K.
\]
\begin{definition}
                    Let $n \geq 3$, let $K\subset \Rn$ be a convex body, and let $L$ be a line passing 
                    through the origin. We denote by $R_L : \Rn \rightarrow \Rn$ the element of $O(n)$ 
                    that acts as the identity on $L$, and sends $x$ to $-x$ on the hyperplane 
                    $L^{\perp}$. The line $L$ is said to be an axis of symmetry of K if the following 
                    relation holds:
\[
                    R_L(K) = K.
\]
                    When the line $L$ does not pass through the origin, we abuse the notation and 
                    denote also by $R_L$ the function that acts as the identity on $L$, and sends 
                    $p + x$ to $p - x$ for every $p\in L$ and $x \in p + L^{\perp}$.
\end{definition}
\begin{remark}\label{soul} 
                    If $L$ is an axis of symmetry of $K$, then, on the one hand, all the sections of $K$ 
                    by hyperplanes orthogonal to $L$ are centrally symmetric with center at $L$; on 
                    the other hand, all the $2$-sections of $K$, by planes containing $L$, have $L$ as 
                    an axis of symmetry.
\end{remark} 
           The line $L$ is said to be an \textit{affine axis of symmetry} 
           (\textit{affine axis of revolution}) of $K$ if there exists an affine 
           transformation $T: \mathbb{A}^n \rightarrow \mathbb{A}^n$ such that 
           $T(L)$ is an axis of symmetry (axis of revolution) of $T(K)$, where 
           $\mathbb{A}^n$ is the real affine space of dimension $n$.
\begin{definition}
           Let $H \subset \mathbb{R}^n$ be a hyperplane. A mapping 
           $S : \mathbb{R}^n \to \mathbb{R}^n$ is a \emph{reflection} with respect to $H$ if, for every 
           point $x \in \mathbb{R}^n$, the point $S(x)$ lies on the line orthogonal to $H$ through $x$, 
           at equal distance from $H$, and on the opposite side of $H$ from $x$. A convex body 
           $K \subset \mathbb{R}^n$ is said to be \emph{symmetric} with respect to $S$ if 
\[
S(K) = K.
\] 
           In other words, $H$ is a hyperplane of symmetry of $K$.
\end{definition}
           
           The hyperplane $H$ is said to be an \textit{affine hyperplane of symmetry} 
            of $K$ if there exist an affine transformation $T: \mathbb{A}^n \rightarrow \mathbb{A}^n$ 
            such that $T(H)$ is a hyperplane of symmetry of $T(K)$.
\begin{figure}[H]\centering
\includegraphics [width=4.5in] {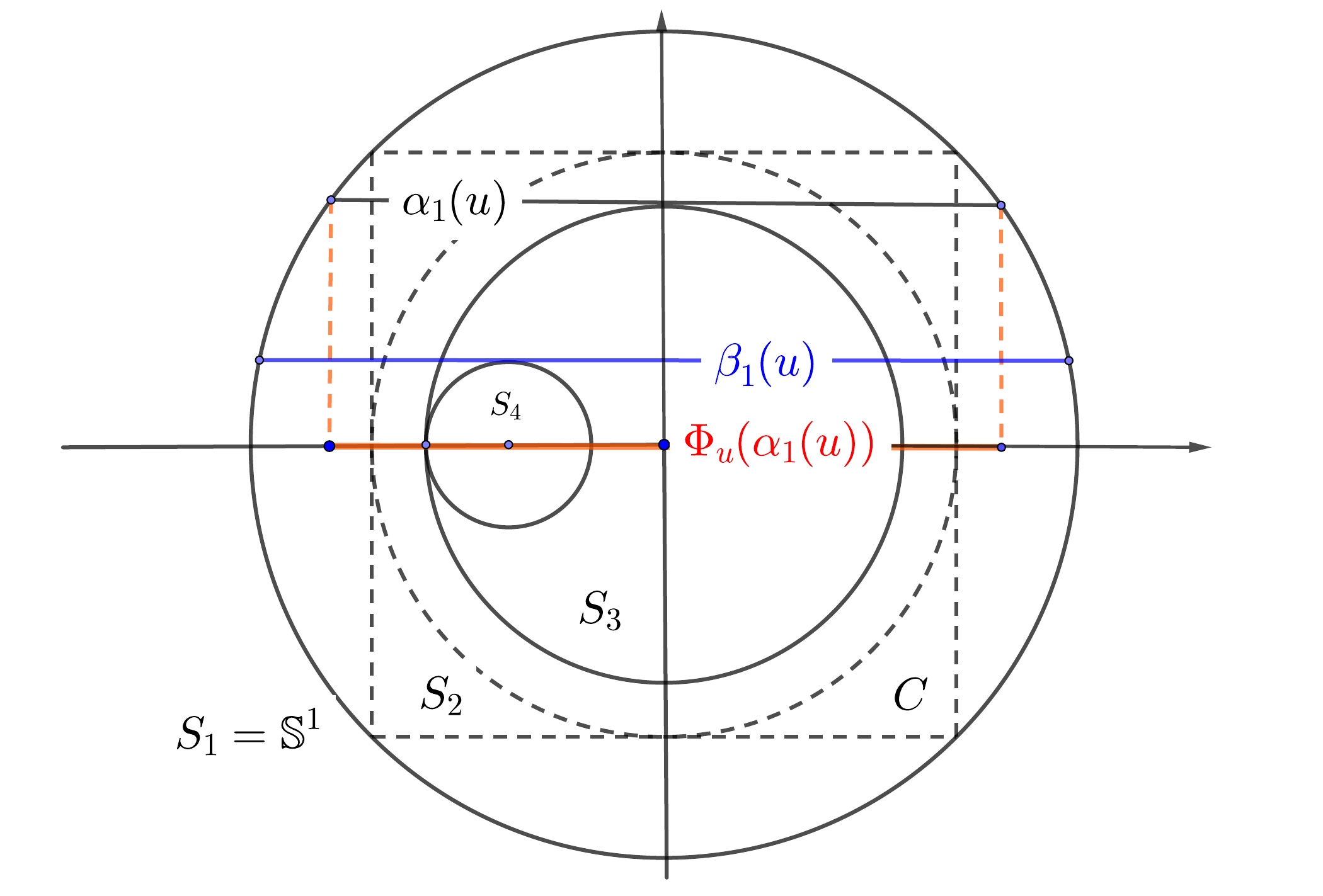} 
\caption{The circle $\mathbb{S}^1$ has suitable circle.}
\label{suit}
\end{figure}
           Let $S_1=\mathbb{S}^1$,  $C$ be a square inscribed in $S_1$ and let $S_2$ be a disc 
           of radius $\frac{1}{\sqrt{2}}$ with center at $O$. Notice that the disc $S_2$ is inscribed in 
           $C$. Let $S_3$ be a disc of radius $<\frac{1}{\sqrt{2}}$ centered at $O$ and let 
           $S_4\subset \inte S_3$ be a disc tangent to $S_3$ such that $O$ does not belong to 
           $S_4$. For $u\in S_1$, we denote by $\alpha_1(u)$, $\alpha_2(u)$ and by $\beta_1(u)$, 
           $\beta_2(u)$ the chords of $S_1$ which are the support chords of $S_3$ and $S_4$, 
           respectively, perpendicular to $u$.
 \begin{claim}\label{bebe}
           The disc $S_4$ is suitable for $S_1$.
\end{claim}
 
            In order to prove the Claim \ref{bebe} we have to show that the condition (\ref{jazzu}) 
            holds for $K=S_1$ and $B=S_4$. This is equivalent to proving that, for all $u\in S_1$, the 
            relation 
\begin{eqnarray}\label{simon}
            \Phi_u(S_4)\subset \Phi_u(\beta_i(u)), \textrm{ }\textrm{ }i=1,2.  
\end{eqnarray}
            holds. Since $\Phi_u(\beta_1(u))=\Phi_u(\beta_2(u))$, will be enough to prove this relation 
            just for $i=1$. By the choice of $S_3$ it is clear that  
\begin{eqnarray}\label{bridge}
            \Phi_u(S_3)\subset \Phi_u(\alpha_1(u)), 
\end{eqnarray}
           (see fig. \ref{suit}). Since $S_4 \subset \inte S_3$, on the one hand, for all $u\in S_1$, 
\begin{eqnarray}\label{apretones}
\Phi_u(S_4)\subset   \Phi_u(S_3)
\end{eqnarray}
and, on the other hand, 
            $|\alpha_1 (u)| \leqslant |   \beta_1(u)|$, where $|A|$ denotes the length of line the segment 
            $A\subset \Rn$. Since $|\Phi_u(\beta_1(u))|=|\beta_1(u)|$ and 
            $|\Phi_u(\alpha_1(u))|=|\alpha_1(u)|$, the previous inequality looks now as 
\begin{eqnarray}\label{grandota}
| \Phi_u(\alpha_1(u))| \leqslant |\Phi_u(  \beta_1(u))|.
\end{eqnarray}
           From (\ref{grandota}) and (\ref{bridge}) it follows $\Phi_u(S_3)\subset \Phi_u(\beta_1(u))$,
           and, from (\ref{apretones}),  $\Phi_u(S_4)\subset \Phi_u(\beta_1(u))$, i.e., the relation 
           (\ref{simon}) holds.
           
           Notice that with the same argument given in the proof of the Claim \ref{bebe}, we can 
           prove that for every solid ellipse $E\subset S_4$ the relations
\begin{eqnarray}\label{perrita}
            \Phi_u(E)\subset \Phi_u(\gamma_i(u)), \textrm{ }\textrm{ }i=1,2.  
\end{eqnarray}
           holds, where $\gamma_1(u)$, $\gamma_2(u)$ are the chords of $S_1$ which are the 
           support chords of $E$ perpendicular to $u$. Therefore we have  
\begin{figure}[H]\centering
\includegraphics [width=4in] {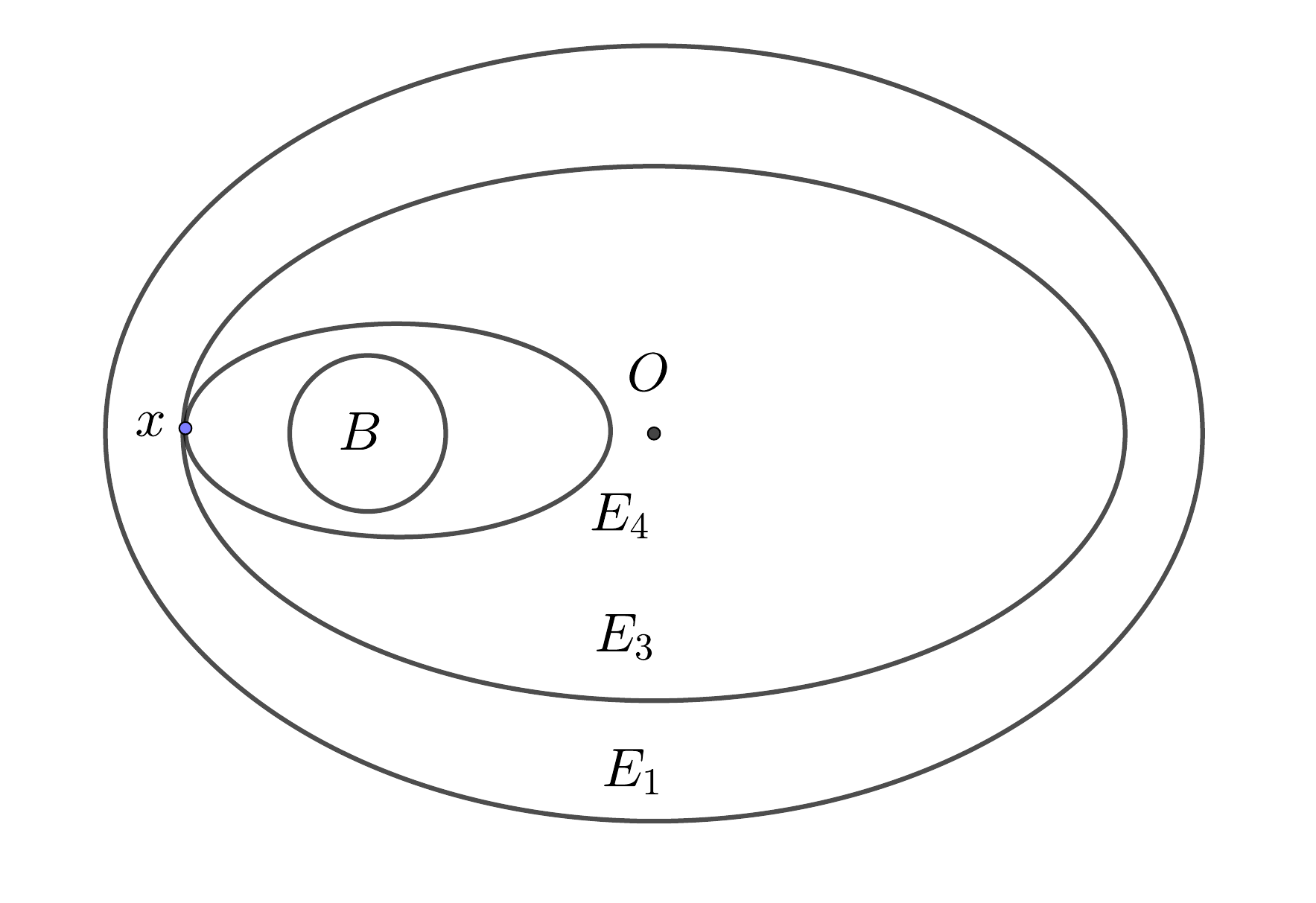} 
\caption{The disc $B$ is suitable for $E_1$.}
\label{terricolas}
\end{figure}
\begin{remark}\label{caderona}
         Let $E_1\subset \Rn$, $n\geq 3$ be a solid ellipsoid with center at $O$. Then the set of 
         suitable balls contained in $E_1$ and not containing $O$ is non-empty.
\end{remark}
\begin{proof}           
           It will be enough to consider the case $n=2$. Let $E_3, E_4\subset \inte E_1$ be solid 
           ellipses, $B\subset E_4$ be a disc and let $x\in E_3$ such that $E_3, E_4$ are homothetic to $E_1$,  
            $E_3$ has center at $O$, $E_4$ is tangent to $E_3$ at $x$, 
            $E_4\backslash \{x\}\subset \inte E_3$ and $O\notin E_4$.      
            (see fig. \ref{terricolas}). Let $T: \mathbb{A}^2 \rightarrow \mathbb{A}^2$ be an affine 
            transformation such that $T(E_1)=S_1$, $T(E_3)=S_3$ and $T(E_4)=S_4$, where $S_1, S_3$ 
            and $S_4$ are discs in accordance with the notation of Claim \ref{bebe}. By virtue of the fact that 
            we have the relation (\ref{perrita}) for the solid ellipse $E:=T(B)$, the disc $B$ is suitable for $E_1$.
\end{proof} 
 \begin{figure}[H]\centering
\includegraphics [width=6.5in] {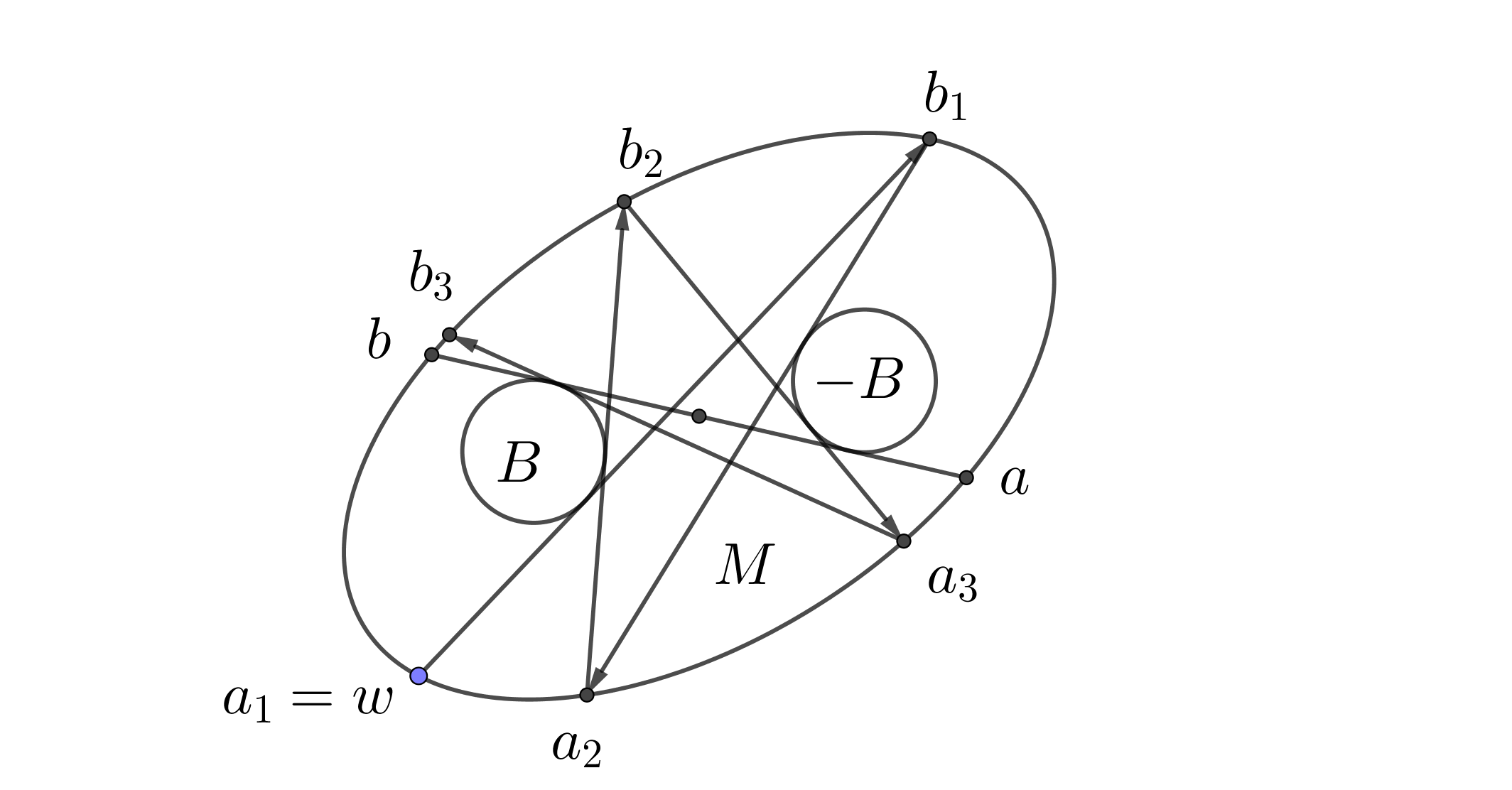} 
\caption{The chords $a_ib_i$ converge to a common chord of $B$ and $-B$.}
\label{chet}
\end{figure}
           \section{Proof of Theorem \ref{mozart} in dimension 3.}
        \subsection{Some Lemmas}
          In this section, we are going to present and prove a series of lemmas which are  
          necessary for the proof of Theorem \ref{mozart} in dimension 3.  
          
                   We recall the strategy for proving Theorem~\ref{mozart}.  We will prove: 
                   \begin{enumerate}
                   \item the body $K$ has many planar shadow boundaries (see Lemmas 
                   \ref{star}, \ref{gato}, \ref{vivaldi}),               
                    \item the body $K$ has many affine axes of symmetry (see Lemma \ref{jamaica}), 
                   \item One of the affine axes of symmetry is, in fact, an affine axis of 
                   revolution (see Lemmas \ref{orchata}, \ref{limon}, \ref{tamarindo}). 
                   \item$K$ has enough affine axes of symmetry to guarantee that the body $K$ is an ellipsoid.
                   \end{enumerate}
 
            Let $M\subset \Rd$ be a centrally symmetric convex body with center 
           at $O$ and let $B$ be a circle such that $O\notin B$. Given a fixed point 
           $w\in \bd K$, we construct two sequences $\{a_i\}_{i=1} ^\infty$, 
           $\{b_i\}_{i=1 }^\infty$ in the following manner. We make $a_1=w$ and 
           for every point $a_i\in \bd M$ we take the point $b_i\in \bd M$ such 
           that the line $L(a_i,b_i)$ is a support line of $B$ and if $u_i$ is a unit 
           vector with the property that the set of vectors $\{b_i-a_i,u_i\}$ is a 
           right frame of $\Rd$, then $B$ is located in the half-space 
           $a_i+\{x\in \Rd:  \langle x,u_i\rangle>0\}$ (See Fig. \ref{chet}). On the 
           other hand, for every $b_i\in \bd M$ we take the point 
           $a_{i+1}\in \bd M$ such that the line $L(b_i,a_{i+1})$ is a supporting 
           line of $-B$ and if $v_i$ is a unit vector with the property that the set of 
           vector $\{a_{i+1}-b_i,v_i\}$ is a right frame of $\Rd$, then $-B$ is 
           located in the half-space $b_{i+1}+\{x\in \Rd:  \langle x,v_i\rangle>0\}$.

           The following lemma will be an important tool in the proofs of Theorem \ref{mozart} in
            dimension 3.  
\begin{lemma}\label{star}
                For the sequences $\{a_i\}_{i=1} ^\infty$, $\{b_i\}_{i=1} ^\infty$ there 
                exists sub-sequences 
                $\{a_{i_s}\}_{s=1} ^\infty$, $\{b_{i_s}\}_{s=1} ^\infty$ which have the 
                properties  
\begin{eqnarray}\label{viri}
a_{i_s} \rightarrow a  \textit{ and } b_{i_s}\rightarrow b, \textit{ when } s\rightarrow \infty
\end{eqnarray} 
                where $a,b\in M$ and $L(a,b)$ is a common support line of $B$ 
                and $-B$ and is passing through $O$. 
\end{lemma}
\begin{proof}
             We will prove the lemma by contradiction. By virtue of the 
             compactness of $M$ there exists a point $a\in \bd M$ and a 
             sub-sequence $\{a_{i_s}\}$ of $\{a_i\}$ such that 
             $a_{i_s}\rightarrow a$, $s\rightarrow \infty$. Let us assume that 
             there exists a point $b\in \bd M$ for which the following two properties 
             hold: 1) the sub-sequence $\{b_{i_s}\}$ of $\{b_i\}$ is such that 
             $b_{i_s}\rightarrow b$, $s\rightarrow \infty$ and 2) $L(a,b)$ is not a 
             common supporting line of $B$ and $-B$ passing through $O$. In 
             particular, suppose that $L(a,b)$ is not a supporting line of $B$. Let 
             $W_1,W_2\subset \Rd$ be two lines parallel to $L(a,b)$ and whose 
             distance to $L(a,b) $ are equal to $\epsilon$, $0<\epsilon<1$. We 
             denote by $D$ the band determined by $W_1,W_2$. Since 
             $a_{i_s}\rightarrow a$, $s\rightarrow \infty$, there exists 
             $N_1\in \mathbb{N}$ such that $|a_{i_s} - a|<\epsilon$ for $s>N_1$. 
             On the other hand, by virtue that $b_{i_s}\rightarrow b$, 
             $s\rightarrow \infty$, there exists $N_2\in \mathbb{N}$ such that 
             $|b_{i_s} - b|<\epsilon$ for $s>N_2$. Let $N:=\max \{N_1,N_2\}$. 
             We take $s_0>N$. Then
\begin{eqnarray}\label{kiss}
             |a_{i_{s_0}}-a|<\epsilon \textrm{ and }|b_{i_{s_0}}-b|<\epsilon,
\end{eqnarray}
              i.e., the chord $[a_{i_{s_0}},b_{i_{s_0}}]$ is contained in $D$.
 
              On the other hand, given $x\in \bd M$ we denote by $l(x),m(x)$ the 
              supporting lines of $B$ passing through $x$ and by $y(x), z(x)$ the 
              intersections of $l(x)$ and $m(x)$ with $\bd M$, respectively, 
              $y(x)\not=x\not= z(x)$. Since $L(a,b)$ is not a supporting line of 
              $B$ we can see that if $\epsilon>0$ is small enough, for each 
              $x\in \bd M$ such that $|x-b|<\epsilon$ the chords $[x,y(x)]$, 
              $[x,z(x)]$ do not belong to $D$. Thus $|a_{i_{s_0}}-a|>\epsilon$ 
              which contradicts (\ref{kiss}).  

              If the sub-sequence $\{b_{i_s}\}$ does not converge, a sub-sequence 
              of it $\{b_{i_{s_{r}}}\}$ will converge to a point $b\in \bd M$. Now we 
              assume again that $L(a,b)$ is not a common supporting line of $B$ 
              and $-B$ passing through $O$. Furthermore, we can assume that 
              $L(a,b)$ is not a supporting line of $B$. Finally, we repeat the 
              previous argument for the sequences 
              $\{a_{i_{s_{r}}}\}$, $\{b_{i_{s_{r}}}\}$ and we will get a contradiction 
              again.
 
              The other cases can be considered analogously.
 \end{proof}
                 Let $A_1, A_2\subset \Rn$ be two sets and let $\Pi$ be a hyperplane. We say the $\Pi$ 
                 \textit{strictly separates $A_1$ and $A_2$} if $A_1$ and $A_2$ are in different open half-spaces 
                  defined by $\Pi$. 
    
                  Let $\Gamma$ be a support plane of $B$. Since $K$ is  centrally 
                  symmetric with center at $O$ and, by hypothesis, $\Gamma \cap K$ is centrally 
                  symmetric, there exists a vector $u$ such that
\begin{eqnarray}\label{delicia}
                 -(\Gamma\cap K)=u+(\Gamma\cap K).
\end{eqnarray}
                   We denote by $P_u$ the plane parallel to $\Gamma$ passing through $O$.
\begin{lemma}\label{gato} 
                   For every support plane $\Delta$ of $B$, parallel to $u$, the center $c_\Delta$ of the 
                    section $\Delta \cap K$ is in $P_u$, equivalently, the vector $v\in \Rt$ with the property 
\[
                   -(\Delta \cap K)=v+(\Delta \cap K)
\]
                   is parallel to $P_u$.  
\end{lemma}
\begin{proof} 
                   The proof is divided in three cases:
    \begin{itemize}
        \item [A)] $\Gamma$ does not separate $B$ and $-B$ and 
                   $\Gamma \cap -B=\emptyset$,    
       \item [B)] $\Gamma$ strictly separates $B$ and
                    $-B$,
       \item [C)] $\Gamma$ does not separate $B$ and $-B$ and 
                   $\Gamma \cap -B\not=\emptyset$.
     \end{itemize}
                   First, we prove case A), i.e., we suppose that $\Gamma$ does not separate $B$ and 
                   $-B$ and $\Gamma \cap -B=\emptyset$. Let $\Delta$ be a supporting plane of $B$ 
                   parallel to $u$, where $u$ satisfies (\ref{delicia}). We are going to prove that the center 
                   $z$ of $\Delta \cap K$ is in $P_u$. Since $B$ is suitable for $K$, the relation (\ref{jazzu}) 
                   holds, consequently, the chords $\Delta \cap (-\Gamma \cap K)$ and 
                   $\Delta \cap (\Gamma \cap K)$ of $\Delta \cap K$ are well defined. Thus, by 
                   (\ref{delicia}) it follows that
\begin{eqnarray}\label{bella}
                   \Delta \cap (-\Gamma \cap K)=  u+[ \Delta \cap (\Gamma \cap K)].
\end{eqnarray}
                   By virtue of (\ref{bella}) we conclude that the chords 
                   $ \Delta \cap (-\Gamma \cap K)$ and $\Delta \cap (\Gamma \cap K)$ of 
                   $\Delta \cap K$ are parallel, and they have the same length. From this and by 
                   the strict convexity of $\Delta \cap K$ and the Barker-Larman condition, it follows that $z\in P_u$.
                   
                   The proofs for the cases B) and C) are completely analogous to the proof of case A) 
                   however we must justify the fact that the chords $\Delta \cap (-\Gamma \cap K)$ and 
                   $\Delta \cap (\Gamma \cap K)$ of $\Delta \cap K$ are well defined in both cases. 
                   
                   Proof for the case B). Let $\Gamma$ be a support plane of $B$ which strictly 
                   separates $B$ and $-B$. We denote by $C$ the infinite cylinder defined by 
                   $K_{\Gamma}$ and $u$, where $u$ satisfies (\ref{delicia}). We are going to prove 
                   that $B\subset C$. Since $\Gamma$ is a support plane of $B$ which strictly 
                   separates $B$ and $-B$ it follows that 
                   $B\subset K\backslash [\conv (K_{\Gamma}\cup K_{-\Gamma})]$ 
                   (notice that $-\Gamma$  strictly separates $B$ and $-B$). By Theorem 12.2.1 P. 297 
                   of \cite{Matousek}, for all plane $W$ parallel to $\Gamma$ such that 
                   $ W \cap [\conv (K_{\Gamma}\cup K_{-\Gamma})]=\emptyset$, we have 
\[
                   W\cap K\subset W\cap C.                   
\]
                   Since $B\subset K\backslash [\conv (K_{\Gamma}\cup K_{-\Gamma})]$ 
                   it follows that
\[
                   W \cap B\subset W\cap K.
\]
                   Hence $W \cap B\subset W\cap C$. The arbitrariness of $W$ it yields our claim 
                   that $B\subset C$. Analogously, we can see that  $-B\subset C$.

                    By virtue that $B\subset C$,  $-B\subset C$, for every support plane $\Delta$ 
                    of $B$, parallel to $u$, the chords $\Delta \cap (\Gamma \cap K)$ and 
                    $ \Delta \cap (-\Gamma \cap K)$ of $\Delta \cap K$ are well defined.
                   
                   Proof for the case C). Now we consider the case when $\Gamma$ does not separate 
                   $B$ and $-B$ and $\Gamma \cap -B\not=\emptyset$. Let $\Gamma'$ be a support 
                   plane of $-B$ such that $\Gamma'\not=-\Gamma$. Let $u,u'\in \Rt$ such the relation 
                   (\ref{delicia}) holds for the planes $\Gamma$, $-\Gamma$ and $\Gamma'$, $-\Gamma'$, 
                   respectively. Let $\Delta$ be a support plane of $B$ parallel to $u'$. By virtue of $B$ 
                   is suitable 
                   for $K$ and $\Gamma'$ is support plane of $G$ then 
\[
                  B\subset \conv (K_{\Gamma'} \cup K_{-\Gamma'}),  
\]
                   consequently, the chords $\Delta \cap (-\Gamma' \cap K)$ and 
                   $\Delta \cap (\Gamma' \cap K)$ of $\Delta \cap K$ are well defined. Furthermore, 
                   notice that $\Gamma'$ satisfies the condition of case A), therefore 
\begin{eqnarray}\label{musica}
c_{\Delta}\in P_{u'}=P_{u}.
\end{eqnarray}
\begin{figure}[H]\centering
        \includegraphics [width=6.0in] {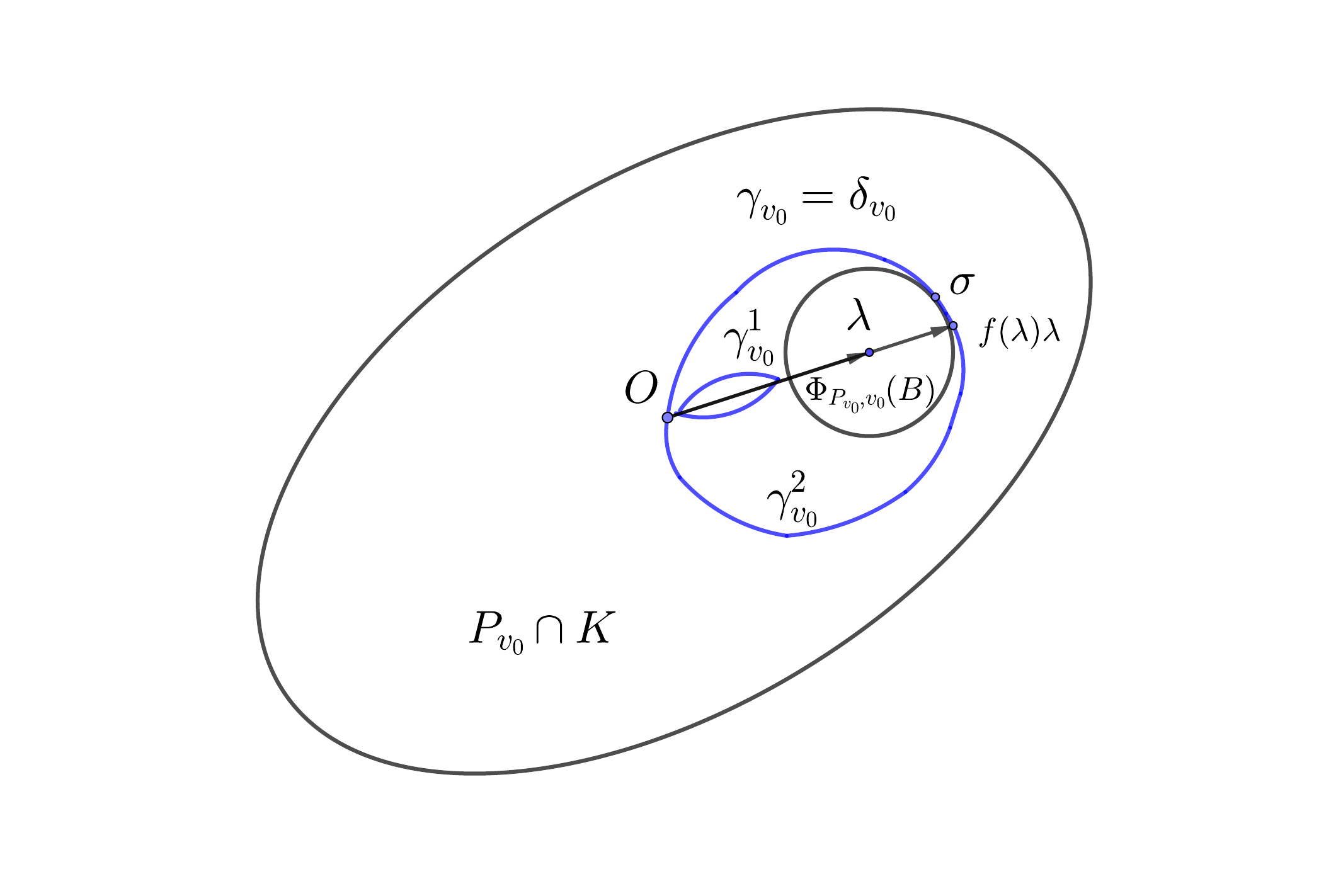} 
        \caption{ The curve $\delta_{v}=\gamma_{v}$ is the union of two simple closed 
        con\-ti\-nuous curves $\gamma_{v}^1$, $\gamma_{v}^2$.}
        \label{pompotas}
\end{figure} 
                   To finish, we need to prove that $u$ and $u'$ are parallel. Since the planes 
                   $\Gamma$, $-\Gamma$ are between the planes $\Gamma'$ and 
                   $-\Gamma'$ (i.e., are in the band defined by the planes $\Gamma'$ and 
                   $-\Gamma'$) the chords $\Delta \cap (-\Gamma \cap K)$ and 
                   $\Delta \cap (\Gamma \cap K)$ of $\Delta \cap K$ are well defined. Because the 
                   distance from the plane $\Gamma$ to $O$ is equal to the distance from the plane 
                   $-\Gamma$ to $O$ it follows that distance from the line $\Gamma\cap \Delta$ to 
                   $c_{\Delta}$ is equal to the distance from the line $-\Gamma \cap \Delta$ to 
                   $c_{\Delta}$. Thus the length of the chords $\Delta \cap (\Gamma \cap K)$ and 
                   $\Delta \cap (-\Gamma \cap K)$ of $\Delta \cap K$ are equal. Analogously, we prove that 
                    the length of the chords $-\Delta \cap (\Gamma \cap K)$ and 
                    $-\Delta \cap (-\Gamma \cap K)$ of $-\Delta \cap K$ are equal. This implies that $u$ 
                   is parallel to $\Delta$, otherwise, we would contradict the strict convexity of the 
                   section $-\Gamma \cap K$ because there would be four parallel chords of 
                   $-\Gamma \cap K$ of the same length $\Delta \cap (-\Gamma \cap K)$, 
                   $-\Delta \cap (-\Gamma \cap K)$, 
                   $u+[\Delta \cap (\Gamma \cap K)]$, $u+[-\Delta \cap (\Gamma \cap K)]$. Varying 
                   the support plane $\Delta$ of $B$ parallel to $u'$ we conclude that $u$ is parallel to $u'$.
\end{proof}
We have the following corollary of Lemma \ref{gato}, 
 \begin{corollary}\label{abejita}                
The locus 
\[
\gamma_u:=\{c_\Delta: \Delta \textrm{ support plane of } B \textrm{ parallel to }u\}\subset P_u
\] 
                    is a planar closed continuous curve and, for every 
                    $\omega \in \mathbb{S}^2\cap P_u$, there exists a real number $f(\omega)$ 
                    such that $  f(\omega)\omega$ 
                    belongs to $\gamma_u$. Furthermore,   
\begin{eqnarray}\label{perro}
\textrm{ if } v\in \gamma_u,  \textrm{ then }u\in \gamma_v  
\end{eqnarray}            
and $\gamma_u$ coincides with the locus
\[
\delta_u:=\{ \textrm{ mid-points of the chords of the section }P_u\cap K \textrm{tangent to }\Phi_{P_u, u}(B)\}.
\]
 \end{corollary}    
\begin{proof}
                    We are going to prove that for every $\omega \in \mathbb{S}^2\cap P_u$, there exists 
                    a real number $f(\omega)$ such that $  f(\omega)\omega$ belongs to $\gamma_u$. 
                    To prove this, since it is easy to prove that $\gamma_u$ coincides with 
                    $\delta_u$, it will be enough to observe that, for the ray $\rho$ defined by 
                    $\omega \in \mathbb{S}^2\cap P_u$, the chords $[\alpha,\beta]$ and 
                    $[\gamma, \delta]$ of the section $P_u\cap K$ 
                    tangent to $\Phi_{P_u, u}(B)$ and parallel to $\rho$ are such that their mid-points, 
                    $ \mu,\tau\in \delta_u=\gamma_u$, respectively, are in different half-planes defined by 
                    $\rho$. The existence of $f(\omega)$ such that $ f(\omega)\omega \in \gamma_u$ 
                    follows from the continuity of the curve $\gamma_u$ by virtue that 
                    $ \mu, \tau \in \gamma_u$.
                     
                    The curve $\delta_{v}=\gamma_{v}$ is the union of two continuous curves 
                  $\gamma_{v}^1$, $\gamma_{v}^2$ such that 
                  $\gamma_{v}^1 \cap \gamma_{v}^2$ consists only of the point $O$, the curves 
                  $\gamma_{v}^1$, $ \gamma_{v}^2$ are simple closed with
                  $\gamma_{v}^1\backslash \{O\}\subset \inte [\gamma_{v}^2]$ (By the Jordan Curve 
                  Theorem \cite{Hales} the set 
                   $\inte [\gamma_{v}^2]$ is well defined), the ellipse 
                  $\Phi_{P_{v}, {v}}(B)$ and the curve $\gamma_{v}^2 \backslash \{O\}$ has a 
                  common point, say $\sigma$, and the relation 
\begin{eqnarray}\label{collelle}
 \Phi_{P_{v}, {v}}(B) \backslash \{\sigma\} \subset \inte [\gamma_{v}^2]
\end{eqnarray} 
                   holds (see fig. \ref{pompotas}). 
         \end{proof}
          
                    For $u\in \mathbb{S}^{n-1}$, we denote by $S\partial(K,u)$ ($S\partial(K,L)$) 
                   the \textit{shadow boundary of $K$ with respect to the vector} $u$ (\textit{to the 
                   line $L$}), i.e., the set of points $x$ in $\bd K$ such that there exists a supporting 
                   hyperplane of $K$ passing through $x$ and parallel to $u$ (parallel to $L$), and 
                   by $u^{\perp}$ the plane perpendicular to $u$ passing through the origin. 
 \begin{lemma}\label{vivaldi}
                   With the notation of Lemma \ref{gato}, if the planes $\Gamma$, $P_u$ and the 
                   vector $u\in \mathbb{S}^{2}$, which satisfies (\ref{delicia}), are such that 
                   $O\notin \Phi_{P_u, u}(B)$, then the shadow boundary $S\partial(K,u)$ is contained 
                   in the plane $P_u$.
\end{lemma}
\begin{figure}[H]
    \centering
    \includegraphics[width=6.5in]{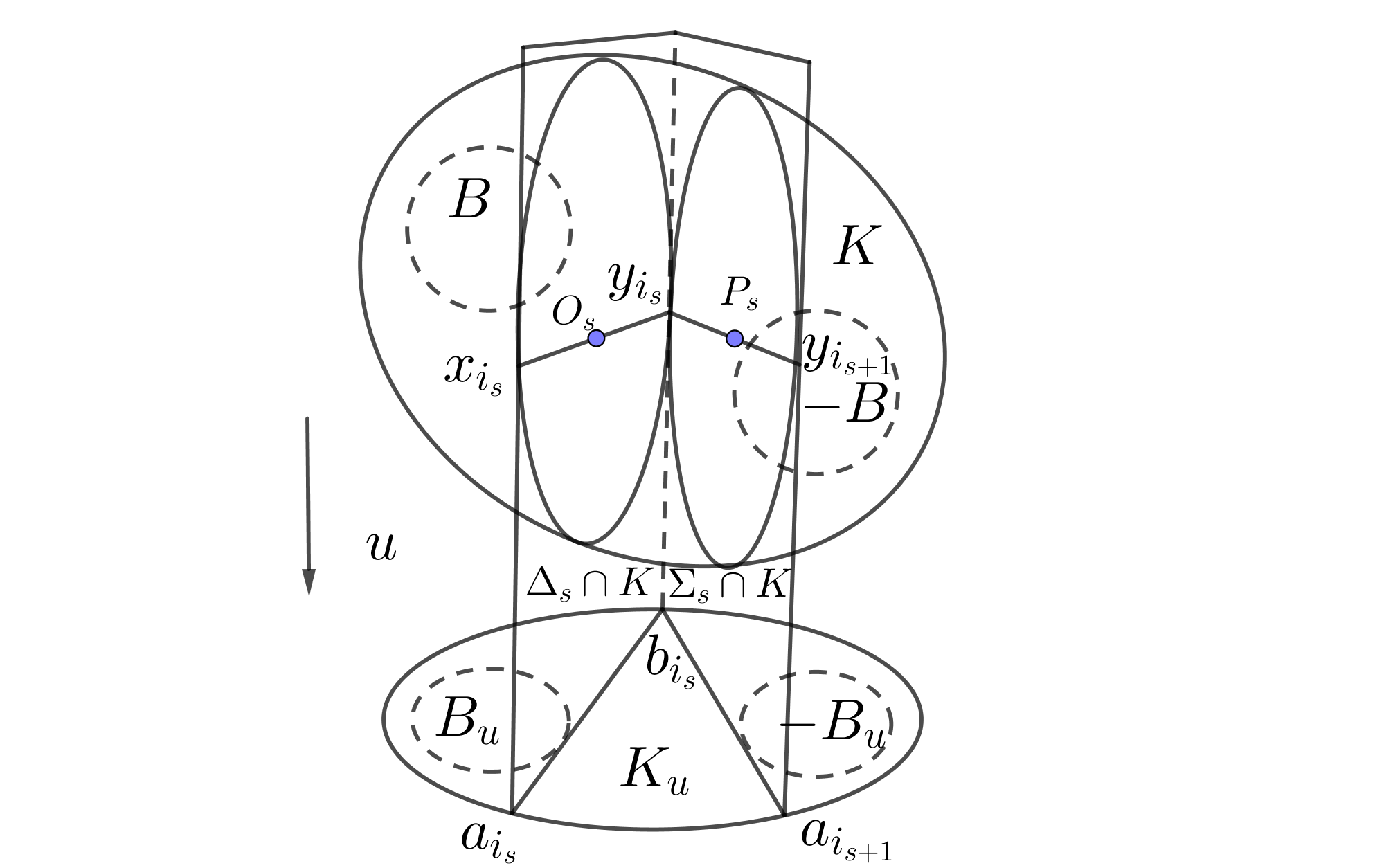}
    \caption{The shadow boundary $S\partial(K,u)$ is contained in a plane.}
    \label{naranja}
\end{figure}  
\begin{proof} 
                   We denote by $\xi (u)$ the shadow boundary $S\partial(K,u)$ of $K$ in the direction 
                   $u$. Let $P_u^1,P_u^2$ the two closed half-spaces defined by $P_u$. We claim 
                   that either $\xi(u) \subset \inte K_1$ or $\xi(u) \subset \inte K_2$ are impossible, 
                   where $K_i:=K\cap P_u^i$,  $i=1,2$.   Since $\xi(u)$ is centrally symmetric with 
                   center at $O$, if $x\in \xi(u)$ and $x\in \inte K_1$, then $-x\in \xi(u)$ and 
                   $-x\in \inte K_2$ (By virtue that $\inte K_2=-\inte K_1$). Thus 
                   $\xi(u)\cap P_u\not=\emptyset$. 
 
                    Let $z\in \xi(u)\cap P_u$. Let $L$ be a supporting line of $K$ through $z$ parallel to 
                   $u$. For $K_u:=\Phi_{P_u,u}(K)$, $B_u:=\Phi_{P_u ,u }(B)$, 
                   $\bar{B}_u:=\Phi_{P_u,u}(-B)$ and $z=a_1$ we construct the sequences $\{a_i\}$, 
                   $\{b_i\}\subset P_u$ as in the Lemma \ref{star}. By virtue of $O\notin \Phi_{P_u, u}(B)$, 
                   We are in a position to apply the Lemma \ref{star}. Thus, there exist sub-sequences 
                   $\{a_{i_s}\}$, $\{b_{i_s}\}$ such that
\begin{eqnarray}\label{richard}
                   a_{i_s} \rightarrow a  \textrm{ and } b_{i_s} \rightarrow b,
\end{eqnarray}
                   when $s\rightarrow \infty$, where $a,b\in \bd K_u$ and $L(a,b)$ is a common tangent 
                   line of $B_u$ and $\bar{B}_u$ and through $O$. It is clear that 
                   $\Delta_i:=\Phi_{P_u,u}^{-1}(L(a_{i},b_{i}))$ is a supporting plane of $B$ and 
                   $\Sigma_i:=\Phi_{P_u,u}^{-1}(L(b_{i},a_{i+1}))$ is a supporting plane of $-B$. By 
                   Lemma \ref{gato},  the centre $o_i$ of the section $\Delta_i \cap K$ and the centre $p_i$ of 
                   the section $\Sigma_i \cap K$ are such that                  
                 \begin{eqnarray}\label{deliciosa} 
                  o_i, p_i \in P_u, i=1,2,...
                  \end{eqnarray}
                   Let $\pi_i:\Delta_i \rightarrow \Delta_i$ and $\rho_i:\Sigma_i \rightarrow \Sigma_i$ 
                   be the central symmetries with respect to $o_i$ and $p_i$, respectively. Then
\[
          b_{i}=\pi_i(a_{i}) \textrm{ }\textrm{ and }\textrm{ }a_{i+1}=\rho_i(b_{i}).
\] 
                   It follows that $a_{i},b_{i} \in \xi(u)$ for all $i$. In particular, It follows that 
                   $a_{i_s},b_{i_s} \in \xi(u)$ for all $s$. By (\ref{richard}), it follows that
\begin{eqnarray}\label{carla}
           a,b\in \xi(u).
\end{eqnarray}
                 Let $x_1\in \xi(u)$, $x_1\not=z$. We are going to prove that $x_1\in P_u$. On the 
                 contrary, let us assume that $x_1\notin P_u$. Let $\delta>0$ the distance from $x_1$ to 
                 $P_u$. For $K_u:=\Phi_{P_u ,u}(K)$, $B_u:=\Phi_{P_u,u}(B)$, 
                 $\bar{B}_u:=\Phi_{P_u,u }(-B)$  
                 and $a_1=\Phi_{P_u,u }(x_1)$ we construct the two sequences 
                 $\{a_i\}$, $\{b_i\}\subset P_u$ as in Lemma \ref{star}. By virtue of 
                 $O\notin \Phi_{P_u, u}(B)$, we are in position to apply Lemma \ref{star}. 
                 Therefore there exists the  sub sequences $\{a_{i_s}\}$, $\{b_{i_s}\}$ such that
\begin{eqnarray}\label{kenni}
       a_{i_s} \rightarrow a  \textrm{ and } b_{i_s} \rightarrow b,
\end{eqnarray}
                 when $s\rightarrow \infty$, where $a,b\in \bd K_u$ and $L(a,b)$ is a common 
                 tangent line of $B_u$ and $\bar{B}_u$ and through $O$. 
  
                  We define the plane $\Delta_i, \Sigma_i$ and the maps $\pi_i$, $\rho_i$ as before. 
                 We define the sequences $\{y_{i}:=\pi_i(x_{i})\}$, 
                 $\{x_{i+1}:=\rho_i(y_{i})\}\subset \bd K$  and notice that 
\[
a_{i}=\Phi_{P_u,u}(x_{i}) \textrm{ and }b_{i}=\Phi_{P_u,u}(y_{i}) \textrm{ }\textrm{ }\textrm{ }\textrm{ or, equivalently, }\textrm{ }\textrm{ }\textrm{ }\{x_{i}\}, \{ y_{i}\}\subset S\partial (K,u)
\]
                   (See Fig. \ref{naranja}).  By the Lemma \ref{gato}, the centre $o_i$ of the section 
                   $\Delta_i \cap K$ is in $P_u$ for all $i$ and the centre $p_i$ of the section 
                   $\Sigma_i \cap K$ is in $P_u$ for all $i$. Thus the distances between $y_{i}$ and 
                   $P_u$ and $x_{i}$ and $P_u$ are equals to $\delta$ for all $i$. By virtue of (\ref{kenni}) 
                   and the compactness of $\bd K$ there exists 
                   \begin{eqnarray}\label{cheve}
                   x,y\in\xi(u)
                   \end{eqnarray}
                    such that 
                   $x_{i_s} \rightarrow x $ and $y_{i_s}\rightarrow y$, when $s\rightarrow \infty$, and 
                   \begin{eqnarray}\label{ahorcada}
                   \Phi_{P_u ,u}(x)=a \textrm{ }\textrm{  and }\textrm{ } \Phi_{P_u,u}(y)=b.
                    \end{eqnarray}
                    Furthermore, the distances between 
                   $x$ and $P_u$ and between $y$ and $P_u$ are equal to $\delta$. Therefore $x\not=a$, 
                   $y\not=b$ (since $a,b\in P_u$ because $z=a_1$ and (\ref{deliciosa}) holds), 
                   there are support lines of $K$ passing through $x$ and $y$ parallel to $u$, by (\ref{cheve}) 
                   and there are support lines of $K$ passing through $a$ and $b$ parallel to $u$, by (\ref{carla}).
                  
                   Note that, since $K$ is strictly convex, for every unit vector $u$, and every support line $L$ 
                      of $K$ parallel to $u$, there exists a unique point $x$ on $L$ that $x$ belongs to the shadow 
                      boundary $S\partial (K,u)$. We have shown that $a\not=x$, $a,x\in S\partial (K,u)$, and that, 
                      by (\ref{ahorcada}), 
                      $a$ and $x$ belong to a support line of $K$ parallel to $u$ ; this clearly contradicts the previous 
                      observation. Consequently, we have a contradiction.
\end{proof}
\begin{figure}[H]\centering
\includegraphics [width=6.0in] {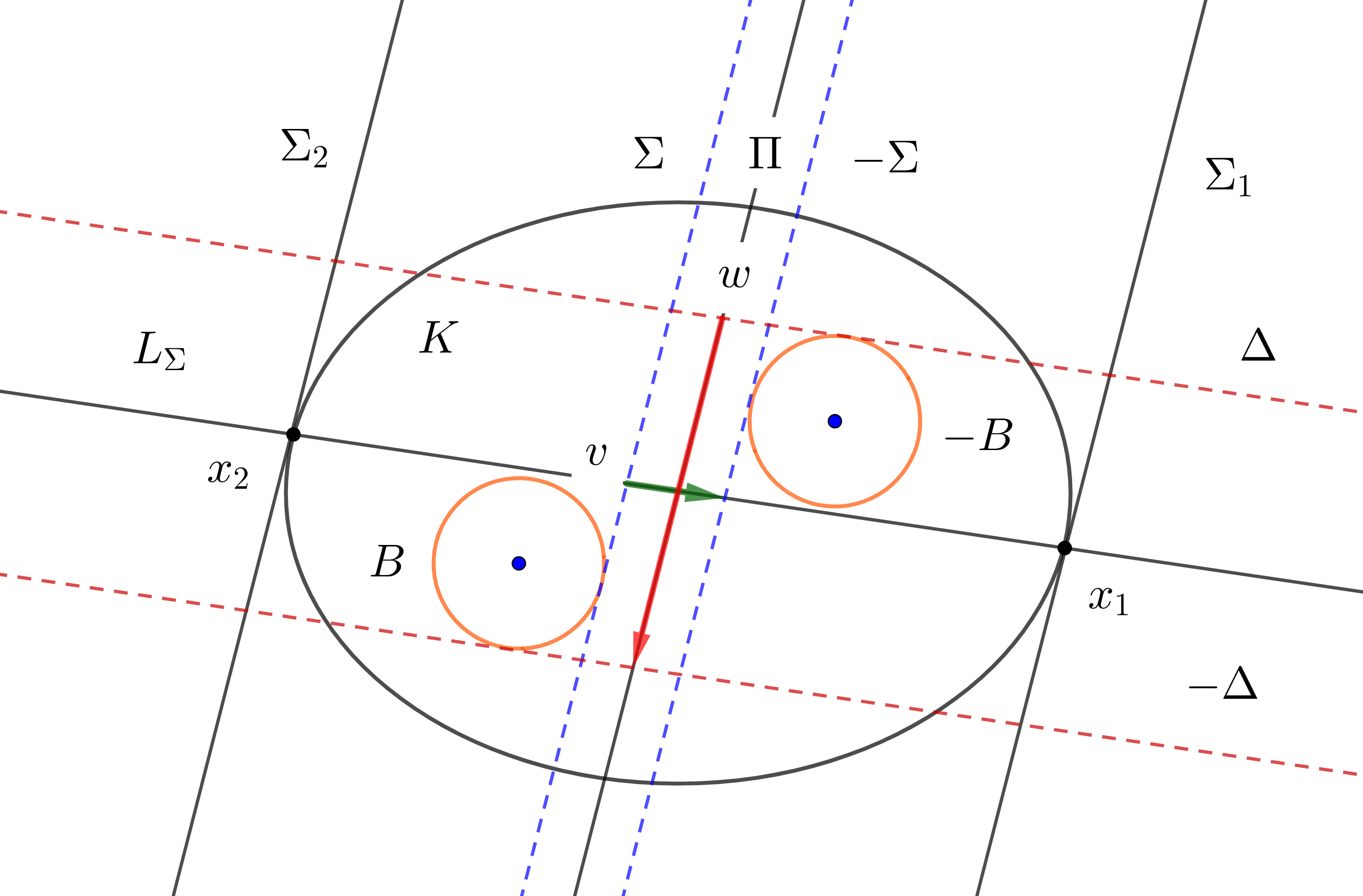} 
\caption{The line $L_{\Sigma}$ is affine axis of symmetry of $K$.}
\label{angel}
 \end{figure}
               Let $\Pi$ be a plane in $\Omega$. Let $\Sigma$ be a support plane 
              of $B$ parallel to $\Pi$, $\Sigma$ separates $B$ and $-B$.  By virtue 
              of $\Sigma \cap K$ has center by hypothesis, the section $-\Sigma \cap K$ has 
              center. Furthermore, there exists $v\in \Rt$ such that 
              $-\Sigma \cap K=v+(\Sigma \cap K)$. Let $\Sigma_1,\Sigma_2$ be 
              the supporting planes of $K$ parallel to $\Pi$. We denote by 
              $L_{\Sigma}$ the line defined by the points 
              $x_i:=\Sigma_i\cap \bd K,i=1,2$.
 \begin{lemma}\label{jamaica}
              The line $L_{\Sigma}$ is an affine axis of symmetry of $K$ parallel 
              to $v$, and all the sections of $K$ parallel to $\Sigma$  are similar 
              and similarly situated.
\end{lemma}
\begin{proof}
               By Lemma \ref{gato}, for every support plane $\Delta$ of $B$, parallel to $v$, the center 
               of the section $\Delta \cap K$ is in $\Pi =P_v$, equivalently, the vector $w\in \Rt$ with the 
               property 
 \begin{eqnarray}\label{tepoz}
                -(\Delta \cap K)=w+(\Delta \cap K) 
\end{eqnarray}      
               is parallel to $\Pi$ (See Fig. \ref{angel}). Consequently, by the choice of $\Pi$, 
               $O\notin \Phi_{\bar{\Delta},w}(B)$, where $\bar{\Delta}$ is a plane parallel to $\Delta$, 
               $O\in \bar{\Delta}$. Thus, by Lemma \ref{vivaldi}, $S\partial(K,w)$ is contained in the 
               plane $\bar{\Delta}$. In particular, since $x_1,x_2\in S\partial(K,w)$ we have that 
               $L_\Sigma\subset \bar{\Delta}$. 

                By a continuity argument, on the one hand, we can prove that varying the support 
                plane $\Delta$ of $B$, always parallel to $v$, the vector $w$ which satisfies (\ref{tepoz}) 
                will take all the directions parallel to $\Pi$ and, on the other hand, for all plane $H$ 
                 containing $L_{\Sigma}$ there exists a support plane $\Delta$ of $B$, parallel to $H$ 
                 (thus parallel to $v$), and a direction $w$ parallel to $\Pi$ which satisfies (\ref{tepoz}), 
                 such that the shadow boundary $S\partial(K,w)$ of $K$, correspon\-ding to the direction 
                 $w$, is given as the section $H\cap K$. 

                 Therefore, all the sections of $K$ parallel to $\Sigma$ are centrally symmetric, with 
                 centre at $L_{\Sigma}$, similar and similarly situated (A proof of the latter can be found 
                 in \cite{larman}).
\end{proof}                   
                 We denote by $\lambda$ the center of $B$.
\begin{lemma}\label{orchata}
                 There exists a support plane $H$ of $B$ such that  the relation
\begin{eqnarray}\label{bonita}
       -H \cap K=-\alpha  \lambda +(H \cap K)
\end{eqnarray}
                 holds, for a real number $\alpha$. Furthermore, $H$ can be chosen such that it strictly separates $B$ and $-B$.  
 \end{lemma} 
 \begin{figure}[H]\centering
\includegraphics [width=6.0in] {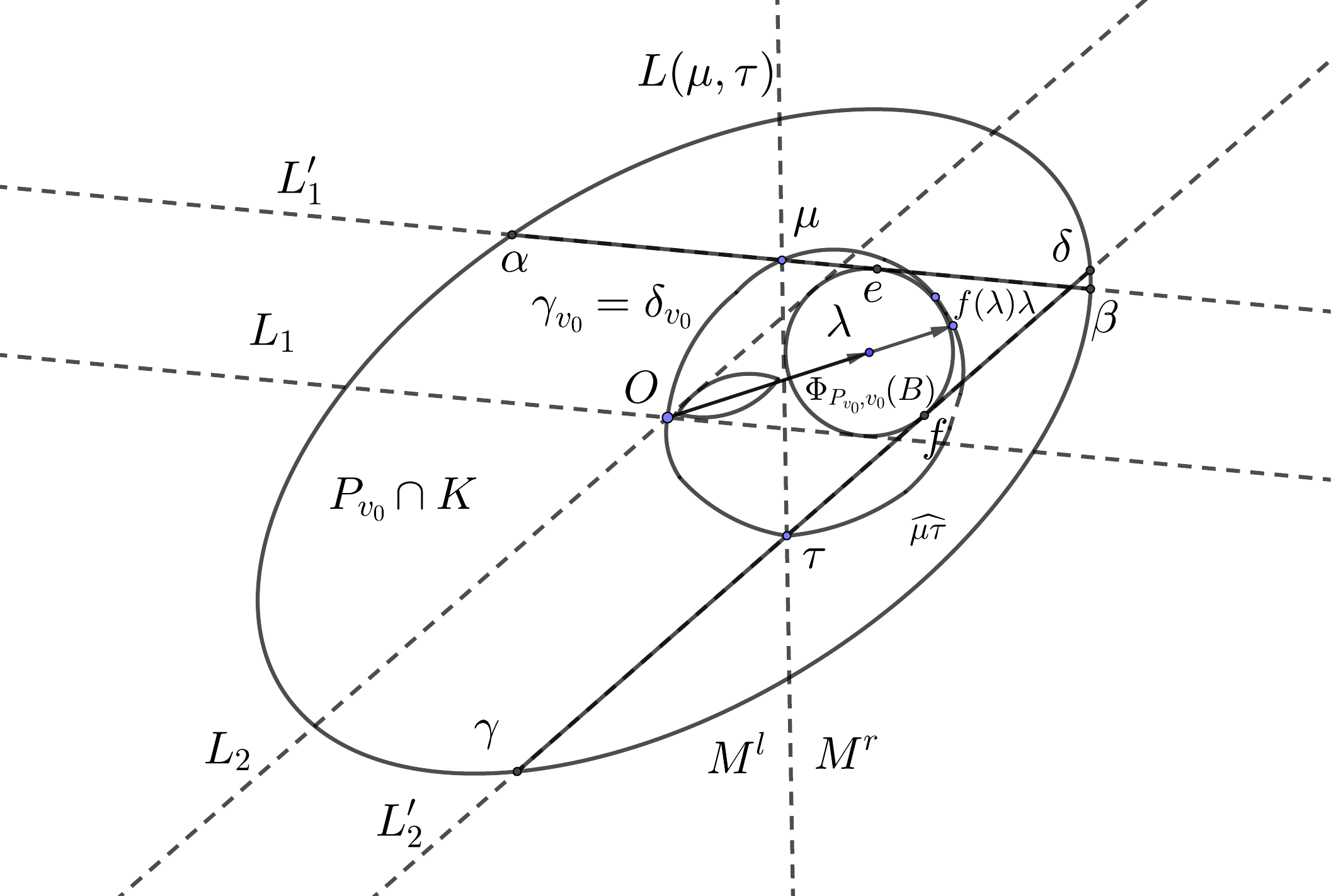} 
\caption{The point $(H\cap P_{v_0})\cap \Phi_{P_{v_0}, {v_0}}(B)$ belongs to $S_2$. }
\label{piernotas}
\end{figure}
 \begin{proof}
                 Let $\Gamma$ be a support plane of $G:=\conv[B\cup (-B)]$. Then there exists 
                 $u\in \Rt$ such that (\ref{delicia}) holds. If $u= \alpha \lambda$ for some real number 
                 $\alpha$, we almost finish, we just need to show that $\Gamma$ can be chosen such that 
                 it strictly separates $B$ and $-B$, which will be seen at the end of the proof of Lemma 
                 \ref{orchata}. Thus we suppose that $u\not= \alpha \lambda$, for every real number 
                 $\alpha$. Notice that, by virtue of the continuity of 
                 $\gamma_u$ (see Lemma \ref{gato}), for all plane $W$ passing through $O$ and 
                 parallel to $u$, there exists $v\in \gamma_u$ such that $W=P_v$. In particular, if we 
                 take $W_0:=\aff\{u,\lambda\}$, there exists $v_0\in \gamma_u$ such that $W_0=P_{v_0}$.
                 By (\ref{perro}), since $v_0\in \gamma_u$ it follows that 
                 $u \in \gamma_{v_0}\subset P_{v_0}= W_0$. By the Corollary \ref{abejita} there 
                 exists a real number $f(\lambda)$ such that $f(\lambda)\lambda \in \gamma_{v_0}$, i.e.,
                 there exists a plane $H$ passing through $O$ and parallel to $v_0$ such that 
                 (\ref{bonita}) holds for a real number $\alpha=f(\lambda)$.
 
                   Next, we are going to prove that we can choose $H$ such that it strictly separates 
                  $B$ and $-B$.  This will be carried out in a series of steps: 1) We describe the curve 
                  $\delta_{v_0}$, 2) we classify the planes $W$ such that $W$ is a support plane of $B$, 
                  parallel to $v_{0}$ and which strictly separates $B$ and $-B$ and 3) we identify where is 
                  the point $\Phi_{P_{v_0}, {v_0}}(H\cap B)$.
  
                  1) This was done at the end of Corollary \ref{abejita}.  
                  
                    2) Let $L_1,L_2$ be the support lines of 
                   $\Phi_{P_{v_0}, {v_0}}(B)$ passing through $O$ and let $L_i'$ be the support line of 
                   $\Phi_{P_{v_0}, {v_0}}(B)$ parallel to $L_i$, $i=1,2$. Let $\mu,\tau$ the mid-points of 
                   the chords $[\alpha, \beta]:=L_1'\cap (P_{v_0}\cap K)$, 
                   $[\gamma, \delta]:=L_2'\cap (P_{v_0}\cap K)$, respectively and let 
                   $e:=[\alpha, \beta] \cap \Phi_{P_{v_0}, {v_0}}(B)$ and 
                   $f:=[\gamma, \delta] \cap \Phi_{P_{v_0}, {v_0}}(B)$. The 
                   points $e,f$ divides the curve $\Phi_{P_{v_0}, {v_0}}(B)$ in two arcs which will be 
                   denoted by $S_1$, $S_2$, we choose the notation such that the reflected points 
                   $ e',f'$ of $e,f$ with respect to the center $\lambda$ of $\Phi_{P_{v_0}, {v_0}}(B)$ are 
                   in $S_1$. We denote by $M^l$ and $M^r$ the half-planes defined by the line 
                   $L( \mu,\tau)$ and we choose the notation such that $\lambda \in M^r$ (notice that 
                   $\lambda\notin L( \mu,\tau)$) (see fig. \ref{piernotas}). The planes $W$ such that $W$ 
                   is a support plane of $B$, parallel to $v_{0}$ and which strictly separates $B$ and $-B$ 
                   are those whose intersection with $P_{v_0}$ defines chords of the section 
                   $P_{v_0}\cap K$ tangent to $\Phi_{P_{v_0}, {v_0}}(B)$ and the intersection point of 
                   such chord with $\Phi_{P_{v_0}, {v_0}}(B)$ is in the arc $\widehat{e'f'}\subset S_1$.
                   
                   3) We proceed as in the proof of Corollary \ref{abejita}. Since $\gamma_{v_0}$ is 
                   a continuous curve, by (\ref{collelle}),  the point 
                   $f(\lambda) \lambda$ belongs to $M^r$, i.e., 
\begin{eqnarray}\label{flaca}
f(\lambda) \lambda \in \widehat{\mu\tau}:=\gamma_{v_0}\cap M^r
\end{eqnarray}                           
                   Let $\nu\in \gamma_{v_o}$, i.e., $\nu$ is the mid-point of the chord $[\zeta,\xi]$ of 
                   $P_{v_0}\cap K$ tangent to $\Phi_{P_{v_0}, {v_0}}(B)$. Notice if $\nu$ is in the arc 
                   $\widehat{\mu\tau}$, then the point $\theta:=[\zeta,\xi]\cap \Phi_{P_{v_0}, {v_0}}(B)$ 
                   belongs to $S_2$. Briefly,
\begin{eqnarray}\label{suculenta}
\textrm{ if }\nu \in \widehat{\mu\tau},\textrm{  then }\theta \in S_2.            
\end{eqnarray}                   
                   Thus, by (\ref{flaca}) and (\ref{suculenta}), the point 
                   $(H\cap P_{v_0})\cap \Phi_{P_{v_0}, {v_0}}(B)$ belongs to $S_2$ (observe that 
                   $(H\cap P_{v_0})\cap \Phi_{P_{v_0}, {v_0}}(B)$ is the mid-point of the chord 
                   $(H\cap P_{v_0})\cap  (P_{v_0}\cap K)$ of the section $P_{v_0}\cap K$ and it is tangent 
                   to the ellipse $\Phi_{P_{v_0}, {v_0}}(B)$).

                   Let $H'$ be a support plane of $B$ parallel to $H$. Let $u'\in \Rt$ such the relation 
                   (\ref{delicia}) holds for the planes $H'$, $- H'$. On the one hand, such as was done in 
                   the proof of the case C) of Lemma \ref{gato}, we prove that $u'$ is parallel to $f(\lambda)\lambda$. On 
                   the other hand, since $(H\cap P_{v_0})\cap \Phi_{P_{v_0}, {v_0}}(B)$ belongs to $S_2$, 
                   then $(H'\cap P_{v_0})\cap \Phi_{P_{v_0}, {v_0}}(B)$ belongs to $S_1$ (see 2)), i.e., 
                   $H'$ strictly separates $B$ and $-B$. The proof of Lemma \ref{orchata} is now complete.
                   \end{proof}
 \begin{corollary}\label{chiquitita}                
                 The line $L_H$ is affine 
                 axis of symmetry of $K$ parallel to $\lambda$ and all the sections of $K$, parallel to 
                 $H$, are similar, similarly situated, with centres at $L_H$. 
 \end{corollary} 
 \begin{proof}               
              It is a direct consequence of Lemma \ref{jamaica}.
 \end{proof}   
 \begin{figure}[H]\centering
\includegraphics [width=6.0in] {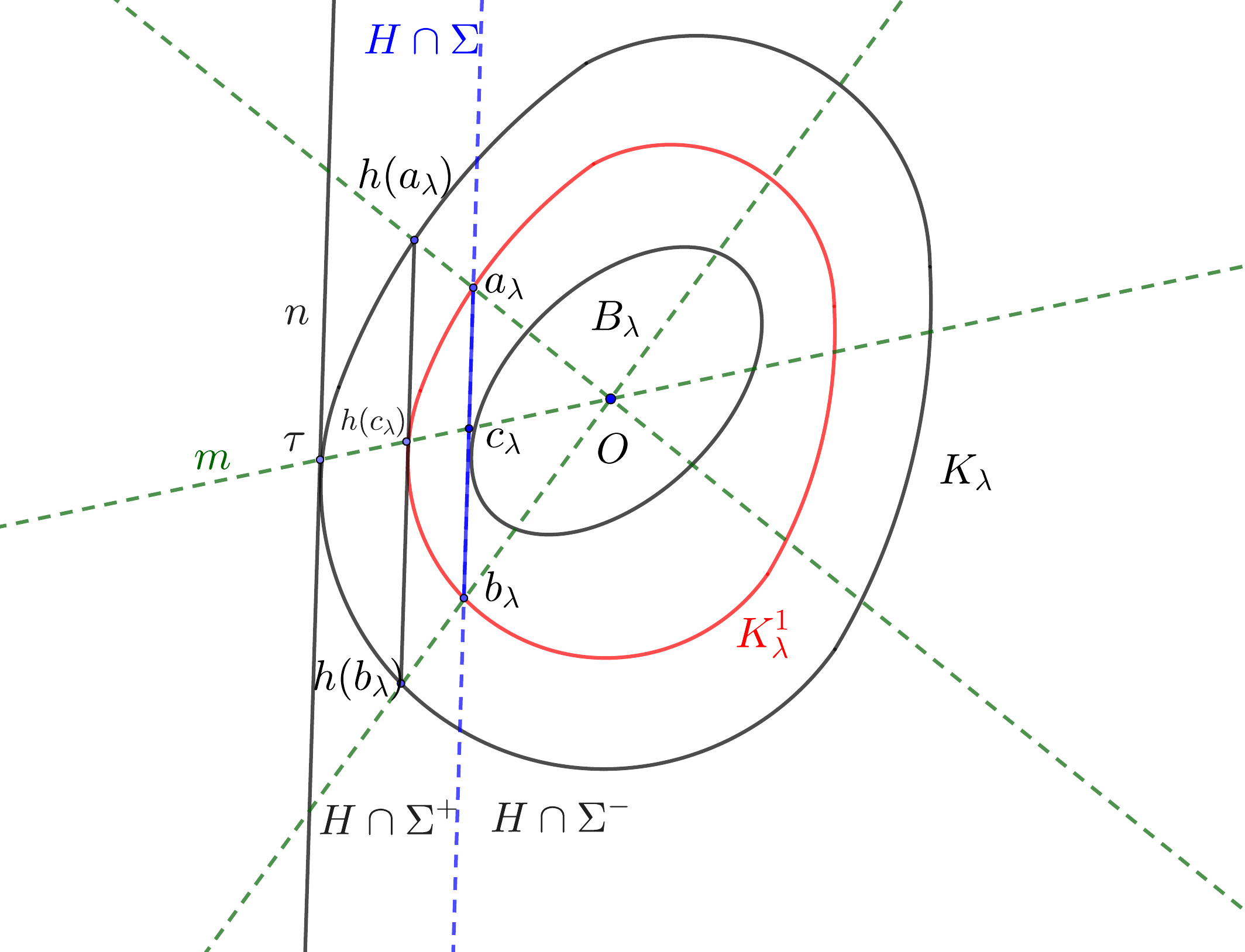} 
\caption{The convex figure 
$K_{\lambda}^{1}$ is the floating body of $K_{\lambda}$.}
\label{dey}
\end{figure}              
                  We consider an affine transformation $A:\Rt \rightarrow \Rt$ such that $A(H)=H$ 
                 and $A(H) \perp A(L_H)$, i.e. $A(L_H)$ is an axis of symmetry of 
                 $A(K)$. We will use the same notation for the geometric objects after applied them 
                 the transformation $A$, that is, we will denote by $K$ the set $A(K)$, by 
                 $L_H$ the set $A(L_H)$ and so on. 
                 
                 The following observation, whose proof is immediate, will be 
                 used several times in the next result, and this is the reason why we present it formally.
\begin{remark}\label{rock}
                 If a convex figure $K\subset \Rd$ is centrally symmetric with center $O$ and it has a 
                 line of symmetry $\mathcal{L}$ passing through $O$, then the line $\mathcal{M}$ 
                 perpendicular to $\mathcal{L}$, passing through $O$, is a line of symmetry of $K$.
\end{remark}

With the notation of Lemmas \ref{gato} and \ref{jamaica} we establish the following lemma.
\begin{lemma}\label{limon} \
 \begin{itemize} 
                  \item [i)] $H$ is a plane of symmetry of $K$. 
                  \item [ii)] for every support plane $\Sigma$ of $B$, parallel to $L_H$, the section 
                  $\Sigma \cap K$ has a line of symmetry $l_{\Sigma}$ 
                  parallel to $L_H$.
\end{itemize}
 \end{lemma}
                   \textit{Proof of} i).
                  Since $L_H$ is axis of symmetry of $K$, by Remark \ref{soul}, for all plane $\Delta$ 
                  containing $L_H$, the section $\Delta \cap K$ has the $L_H$ as  line of 
                  symmetry. On the other hand, since $K$ is centrally symmetric with center $O$, 
                  for all planes $\Delta$ containing $L_H$, the section $\Delta \cap K$ is 
                  centrally symmetric with center $O$. Then, for 
                  all plane $\Delta$ containing $L_H$, by Remark \ref{rock}, the line 
                  perpendicular to $L_H$ passing through $O$ is a line of symmetry. Varying 
                  $\Delta$, $L_H\subset \Delta$, it follows that $H$ is plane of symmetry of 
                  $K$.
                 
   \textit{Proof of} ii).
                Let $\Sigma$ be a support plane of $B$ parallel to $L_H$. By hypothesis 
                 $\Sigma \cap K$ is centrally symmetric with centre $c_{\Sigma}$. On the one hand, 
                 since $L_H$ is parallel to $\lambda$, by Lemma \ref{jamaica}, $\Sigma$ is 
                 parallel to $\lambda$, then, by Lemma \ref{gato}, $c_{\Sigma}\in H$. Hence 
                 $c_{\Sigma}\in \Sigma \cap H$. On the other hand, by virtue that $H$ is plane of 
                 symmetry of $K$ and $\Sigma$ is parallel to $L_H$, the section $\Sigma \cap K$
                  has the line $\Sigma \cap H$ as a line of symmetry. Thus, by Remark \ref{rock}, 
                  $\Sigma \cap K$ has a line of symmetry $l_{\Sigma}$ parallel to $L_H$.
                 \fin
                 
                 The following notion will be important in the proof of the next lemma. According to 
                 \cite{werner}, the \textit{convex floating body} $K_{\delta}$ of a convex body $K$ in 
                 $\Rn$ is the intersection of all halfspaces whose hyperplanes cut off a set of volume 
                 $\delta$ from $K$, for our purposes it is enough to take $\delta$ such that it satisfies 
                 the inequalities $0<\delta < \frac{\textrm{Vol } (K)}{2}$. Notice that the midpoint of 
                 every chord of $K$ which is tangent to $K_{\delta}$ belongs to $K_{\delta}$ (see for 
                 instance \cite{schutt}). 
\begin{lemma}\label{tamarindo}
        $K$ is affine equivalent to a body of revolution.    
\end{lemma}
 \begin{proof}
                   Since $H$ is plane of symmetry of $K$ the relation
\begin{eqnarray}\label{manguito}
         S\partial(K,L_H)=H\cap \bd K    
\end{eqnarray}
                  holds. By (\ref{manguito}) we conclude that  the projection 
                 $K_{\lambda}:=\Phi_{\lambda}(K)$ 
                  is equal to the section $H\cap \bd K$. Now we are going to prove that the projection 
                  $K_{\lambda}$ is an ellipse. After this, by Corollary \ref{chiquitita}, we would have that all the sections   
                  $H_1\cap K$ are homothetic ellipses and with centers at $L_H$, where $H_1$ is a plane parallel to $H$ and such 
                  that $H_1\cap \inte K\not= \emptyset$, i.e., $L_H$ is an affine axis of revolution of $K$.
                  
                   Let $\Sigma$ be a plane parallel to $L_H$ and tangent to $B$. We denote by 
                  $\Sigma^{+}$ and $\Sigma^{-}$ the half-spaces defined by $\Sigma$, choosing the 
                  notation such that $O\in \Sigma^{-}$. Let $H_1$ be a plane parallel to $H$ and such 
                  that 
\begin{eqnarray}\label{ale}
                  H_1\cap \inte (\Sigma \cap K)\not= \emptyset.
\end{eqnarray}
                  We denote by $a,b$ the extreme points of the chord $ H_1 \cap (\Sigma  \cap K)=\Sigma  \cap (H_1\cap K)$. 
                  Let $W_{\Sigma}$ be the plane containing the parallel lines $l_{\Sigma}$ and 
                  $L_H$. Since $l_{\Sigma}$ is line of symmetry of $ \Sigma \cap K$ the chord 
                  $ab$ has his mid-point $c$ at $l_{\Sigma}\subset W_{\Sigma}$. It follows that the 
                  chord with extreme points $a_{\lambda}:=\Phi_{ \lambda}(a)$,
                  $b_{\lambda}:=\Phi_{\lambda}(b)$ of $K_{\lambda}^{1}:=\Phi_{\lambda}(H_1 \cap K)$ 
                  has its mid-point $c_{\lambda}:=\Phi_{\lambda}(c)$ at $m:=W_{\Sigma}\cap H$ and 
                  it is tangent to the ellipse $B_{\lambda}:=\Phi_{\lambda}(B)=\Phi_{\lambda}(-B)$ (See Fig. \ref{dey}). 
                  Notice that the last equality is due to the choice of $H$ (see relation (\ref{bonita})) and the fact that 
                  $L_H$ is parallel to $\lambda$. 

                 By Corollary \ref{chiquitita}, $H\cap K=K_{\lambda}$ and $H_1 \cap K$ are homothetic and centrally 
                 symmetric with centers at $L_H$, then $K_{\lambda}$ and $K_{\lambda}^{1}$ 
                 are homothetic with center of homothety $O$. Let $h:H \rightarrow H$ be the 
                 homothecy, with center of homothety $O$, such that 
                 $h(K_{\lambda}^{1}))=K_{\lambda}$. We conclude that the chord 
                 $h(a_{\lambda})h( b_{\lambda})$ of $K_{\lambda}$ has its mid-point at $m$. 

                 Varying $H_1$, always parallel to $H$ and such that the relation (\ref{ale}) holds, 
                 We get to the following conclusions: 
\begin{itemize}
                  \item [1)] All the chords of $K_{\lambda}$, parallel to the line $H \cap \Sigma$ and 
                  contained in the half-plane $H \cap \Sigma^{+}$, have its mid-point at the line $m$ 
                  (See Fig. \ref{dey}). 
                  \item [2)] There exist a supporting line $n$ of $K_{\lambda}$ parallel to the line 
                  $H \cap \Sigma$ at the point $\tau$, where $\tau \in m \cap \bd K_{\lambda}$ and 
                  $\tau \in H \cap \Sigma^{+}$. 

                   \item [3)] There exist a supporting line $w$ of $K_{\lambda}^1$ at the point $r$ such 
                   that $w$ is parallel to the line $H \cap \Sigma$ and the chord $w\cap K_{\lambda}$ 
                   has its mid-point at $r$. 

                   \item [4)] For all plane $H_1$, parallel to $H$ and close enough to $H$, the 
                   convex figure $K_{\lambda}^{1}$ is the floating body of $K_{\lambda}$.
\end{itemize}
                   By 4) and Theorem 1 of \cite{werner} it follows that $K_{\lambda}$ is an ellipse. 
\end{proof}

\subsection{Proof of Theorem \ref{mozart}.}
                In this section, we are going to prove Theorem \ref{mozart} in dimension 3. We denote 
                by $A(3)$ and $O(3)$ the sets of all the affine and orthogonal transformations from 
                $\Rt$ to $\Rt$, respectively. By virtue of Lemma \ref{tamarindo} there exists  $A\in A(3)$ 
                such that $G:=A(L_H)$ is an axis of revolution of $\bar{K}:=A(K)$. By Lemma 
                \ref{jamaica}, for $\Sigma \in \Omega$, the line $L_{\Sigma}$ is an affine axis of 
                symmetry of $K$. Thus the line $J_{\Sigma}:=A(L_{\Sigma})$  is an affine axis of 
                symmetry of $\bar{K}$.
                Let 
\[
\Lambda_{\Sigma}=\{T(J_{\Sigma}):T\in O(3) \textrm{ such that  } T(G)=G\}.
\]
                For each line $Q\in \Lambda_{\Sigma}$, on the one hand, let $A_Q\in A(3)$ such that 
                $\bar{Q}=A_Q(Q)$ is an axis of symmetry of $\bar{K}_{Q}:=A_Q(\bar{K})$ (and, 
                consequently, $A_{Q}(G)$ is an affine axis of revolution of $\bar{K}_{Q}$). On the other 
                hand, let $R_{\bar{Q}}\in O(3)$ be the rotation with axis $\bar{Q}$ by an angle $\pi$. 
                Thus $R_{\bar{Q}}(\bar{K}_{Q})=\bar{K}_{Q}$. Let 
\[
\mathcal{A}_{\Sigma}=\{A_Q \in A(3):Q\in \Lambda_{\Sigma}\},
\textrm{    }\textrm{    } \mathcal{B}_{\Sigma}=\{R_{\bar{Q}} \in O(3): Q\in \Lambda_{\Sigma}\}  \textrm{ }
\]
and
\[ 
 \textrm{ }\mathcal{C}_{\Sigma}=\{A_{Q}^{-1}(R_{\bar{Q}}(A_{Q}(G))):Q\in \Lambda_{\Sigma}\}.
\]
                  We can illustrate these definitions with the following diagram:
\[
K \stackrel{A }{\longrightarrow} \bar{K} \stackrel{A_{Q} }{\longrightarrow} \bar{K}_{Q} \stackrel{R_{\bar{Q}} }{\longrightarrow} \bar{K}_{Q}  \stackrel{ A^{-1}_{Q} }{\longrightarrow} \bar{K}  \stackrel{A^{-1} }{\longrightarrow} K 
\]
                  Suppose that there exists $\Sigma_1\in \Omega$ such that, for all 
                  $Q\in \Lambda_{\Sigma_1}$, $ A_Q=\textrm{id}$ and $\Lambda_{\Sigma_1}$ is the 
                  \textit{equator} of $\bar{K}$. Thus $\mathcal{C}_{\Sigma_1}$ consist just of the element 
                  $G$. On the one hand, $ A_Q(G)=\textrm{id}(G)=G$, and, on the other hand, 
                  by virtue that $G$ and $Q$ are orthogonal, $R_{\bar{Q}}(G)=G$. Hence
\[
A_{Q}^{-1}(R_{\bar{Q}}(A_{Q}(G)))=A_{Q}^{-1}(R_{\bar{Q}}(\textrm{id}(G)))=A_{Q}^{-1}(R_{\bar{Q}}( G))=A_{Q}^{-1}(G)=G.
\]
                  Let $\Sigma_2\in \Omega$, $\Sigma_1\not=\Sigma_2$. Suppose now that $\Sigma_2$ 
                  is such that $J_{\Sigma_2}:=A(L_{\Sigma_2})$ is an axis of symmetry of $\bar{K}$ (i.e., 
                  $ A_{\Sigma_2}=\textrm{id}$) such that $J_{\Sigma_2}$ does not belong to the equator 
                  $\Lambda_{\Sigma_1}$. Then $\bar {K}$ is a sphere by virtue of having two different 
                  axes of revolution, namely, $G$ and   $R_{J_{\Sigma_2}}(G)$. Consequently, $K$ is 
                  an ellipsoid.
 
                  Now we assume that $J_{\Sigma_2}:=A(L_{\Sigma_2})$ is an affine axis of symmetry, 
                  i.e., $ A_{\Sigma_2}\not=\textrm{id}$. We denote by $\tau$ the collection of planes 
                  containing $G$. For each $\Pi \in \tau$, we will call the set $\Pi \cap \bar{K}$ a 
                  \textit{meridian} of $\bar{K}$. For each $Q \in  \Lambda_{\Sigma_2}$, we denote by 
                  $\zeta_Q$ the meridian passing through $Q$. Since, for every 
                  $Q\in \Lambda_{\Sigma_2}$, the relation 
\[
A_{Q}^{-1}(R_{\bar{Q}}(A_{Q}(\bar{K})))=A_{Q}^{-1}(R_{\bar{Q}}(\bar{K}_Q))=A_{Q}^{-1}( \bar{K}_Q)=
A_{Q}^{-1}(A_Q(\bar{K}))=\bar{K} 
\]
                  holds, every element of $\mathcal{C}_{\Sigma_2}$ is an affine axis of revolution of 
                  $\bar{K}$. On the other hand, notice that for $Q\in \Lambda_{\Sigma_2}$,  
                  $A_{Q}^{-1}(R_{\bar{Q}}(A_{Q}(G)))\subset \zeta_Q$. Consequently, $\bar{K}$ has an 
                  infinite number of affine axes of revolution. By Theorem 1 of \cite{espanol}, the only 
                  convex body with an infinite number of affine axes of revolution is the ellipsoid. Hence 
                  $\bar{K}$ is an ellipsoid. Therefore, $K$ is an ellipsoid.  

                  If $\mathcal{C}_{\Sigma_1}$ no consist of just one element, then $K$ has an infinite 
                  number of affine axes of revolution (see the previous paragraph) and, again by Theorem 
                  1 of \cite{espanol}, it  follows that $K$ is an ellipsoid.

                  
\section{Proof of Theorem \ref{mozart} in dimension $>3$.}
                  We will assume that Theorem \ref{mozart} holds in dimension $n-1$ and we will prove 
                  that it holds in dimension $n$, $n\geq 4$. We take a system of coordinates such $O$ is 
                  the origin. Let $\Pi$ be a hyperplane that separates $B$ from $O$. 

\begin{lemma}\label{chiquita}
                  For every vector $u\in \mathbb{S}^{n-1}$ parallel to $\Pi$ the projection $\Phi_u(K)$ is 
                  an $(n-1)$-ellipsoid.
\end{lemma}
\begin{proof}
                  We claim that the $O$-symmetric strictly convex body $K':=\Phi_u(K)$ satisfies the 
                  Barker-Larman conditions with respect to sphere $B':=\Phi_u(B)$. By the choice of $\Pi$, 
                  notice that $O\notin B'$. Every $(n-2)$-section of $ K'$ given by support 
                  $(n-2)$-plane of $B'$ can be considered as the projection of one $(n-1)$-section 
                  of $K$ given by a support hyperplane of $B$ and pa\-ra\-llel to $u$. Since every 
                  orthogonal projection of a centrally symmetric body is a centrally symmetric body in 
                  dimension $n-1$, $K'$ satisfies the Barker-Larman conditions with respect to sphere $B'$.
                  
                  In order to show that $B'$ is suitable for $K'$ let $\Gamma'$ be a support $(n-2)$-plane of 
                   $G':=\conv[ B'\cup (-B')]$, we must prove that the relation
\begin{eqnarray}\label{mana}
                  B'\subset \conv (K'_{\Gamma'} \cup K'_{-\Gamma'} )  
\end{eqnarray} 
                  holds, where $K'_{\Gamma'}:=\Gamma' \cap K' $ and 
                  $K'_{-\Gamma'}:=-\Gamma' \cap K $. Let $\Gamma$ be a support hyperplane of $B$ 
                  parallel to $u$ such that $\Phi_u(\Gamma)=\Gamma'$. By virtue of $B$ is suitable 
                  for $K$ the relation (\ref{jazzu}) holds. Consequently,  
\[                  
                  B'=\Phi_u( B)\subset \Phi_u(\conv (K_{\Gamma} \cup K_{-\Gamma} )).  
\]                
                  Rewriting the right-hand side of this inclusion, noting that 
                  $\Phi_u(K_{\Gamma})=K'_{\Gamma'}$ and $\Phi_u(K_{-\Gamma})=K'_{-\Gamma'}$, it follows that
\[
\Phi_u(\conv (K_{\Gamma} \cup K_{-\Gamma}))=\conv (\Phi_u(K_{\Gamma}) \cup \Phi_u(K_{-\Gamma}))=\conv (K'_{\Gamma'} \cup K'_{-\Gamma'}). 
\]
                  Therefore, 
\[
B'\subset \conv (K'_{\Gamma'} \cup K'_{-\Gamma'}),
\]                
                  i.e., the relation (\ref{mana}) holds. 
                  
                  By virtue of the induction hypothesis, $\Phi_u(K)$ is an $(n-1)$-ellipsoid.     
\end{proof}
                   Let $\Pi$ be a hyperplane that separates $B$ from $O$. Let $\Gamma$ be a 
                   hyperplane parallel to $\Pi$, $O\in \Gamma$. Let $\Pi_1,\Pi_2$ be support hyperplanes 
                   of $K$ at the points $x_1,x_2\subset \bd K$, respectively, parallel to $\Pi$. Let 
                   $L:=L(x_1, x_2)$.
\begin{lemma}\label{u2}
$\Gamma$ The relation
\begin{eqnarray}\label{jacki}
\Gamma \cap \bd K=S\partial(K,L)
\end{eqnarray}
holds.
 \end{lemma}
\begin{proof}
                    WLG We can assume that $L\perp \Pi$. Let $z\in \Gamma \cap \bd K$ and let $\Delta$ 
                    be a supporting hyperplane of $K$ at $z$. Let $w$ be a unit vector parallel to 
                    $\Delta \cap \Pi$. Let 
\[
  K':=\Phi_w(K), z':=\Phi_w(z), \Gamma':=\Phi_w(\Gamma), \Pi'_i:=\Phi_w(\Pi_i), i=1,2.
\]  
                    By Lemma \ref{chiquita}, $K'$ is an ellipsoid. Since $\Phi_w(L)=L$ and $\Gamma'$ is 
                    parallel to the support $(n-2)$-planes $\Pi'_1$, $\Pi'_2$ of $K'$,   $O\in \Gamma'$, it 
                    follows that
\begin{eqnarray}\label{barbi}
S\partial (K',L)= \Gamma' \cap K'.    
\end{eqnarray}
                    Thus there exists a unique support $(n-2)$-plane $\Sigma$ of $K'$ at $z'$ parallel to 
                    $L$. Consequently, $\Delta=\Phi_{w}^{-1}(\Sigma)$ is parallel to $L$. Therefore 
                    $z\in S\partial (K,L)$. Hence $\Gamma \cap \bd K \subset S\partial (K,L)$.

                    On the other hand, let $z\in S\partial (K,L)$. Let $\Delta$ be a supporting hyperplane 
                    of $K$ at $z$ parallel to $L$. Let $w$ be a unit vector parallel to $\Delta \cap \Pi$. 
                    We use the same notation as above. By Lemma \ref{chiquita}, $K'$ is an ellipsoid. 
                    By virtue that $\Phi_w(\Delta)$ is parallel to $L$ and since the relation (\ref{barbi}) 
                    holds, it follows that $z'\in \Gamma'\cap K'$. Thus $z\in \Gamma \cap \bd K$, that is, 
                    $S\partial (K,L)\subset \Gamma \cap \bd K$.
\end{proof}
\begin{lemma}\label{teta}
                    There exists an affine transformation $A:\Rn \rightarrow \Rn$ such that 
                    $A(\Gamma)=\Gamma$, $A(L)\perp \Gamma$ and the ellipsoids 
                    $\{\Phi_u(A(K)): u\in (\mathbb{S}^{n-1}\cap \Gamma) \}$ are 
                    congruent ellipsoids of revolution with axis $A(L)$, that is, $A(K)$ is an ellipsoid of 
                    revolution with axis $A(L)$. 
\end{lemma}
\begin{proof}
                    For all $u\in (\mathbb{S}^{n-1} \cap \Gamma)$, let $K':=\Phi_u(K)$, 
                    $\Gamma':=\Phi_u(\Gamma)$. Using (\ref{jacki}) and (\ref{barbi}) of Lemma \ref{u2}, 
                    we have
\begin{eqnarray}\label{nalga}
\Phi_u(\Gamma \cap \bd K)=\Phi_u(S\partial(K,L))=S\partial(K',L)=\Gamma'\cap K'.    
\end{eqnarray} 
                    It follows that the $(n-1)$-section $\Gamma \cap \bd K$ is an ellipsoid because all 
                    its orthogonal projection are $(n-2)$-ellipsoids (notice that, by virtue of Lemma 
                    \ref{chiquita}, $\Gamma'\cap K'$ is a $(n-2)$-ellipsoid). Now we choose an affine 
                    transformation $A:\Rn \rightarrow \Rn$ such that  $A(\Gamma)=\Gamma$, 
                    $A(L)\perp \Gamma$, and $A(\Gamma \cap \bd K)$ is a sphere of radius $R$. Thus, 
                    for all $u\in (\mathbb{S}^{n-1} \cap \Gamma)$, $\Phi_u(A(K))$ is an 
                    ellipsoid of revolution with axis $A(L)$ (with the maximum radius of the spheres 
                     perpendicular to $A(L)$ equal to $R$).
\end{proof}
                    \textbf{Proof of Theorem \ref{mozart} in dimension $>3$.}
                     Let $\Pi$ be a hyperplane that separates $B$ from $O$. By Lemma \ref{chiquita}, 
                     for every vector $u\in \mathbb{S}^{n-1}$ parallel to $\Pi$ the projection $\Phi_u(K)$ 
                     is an $(n-1)$-ellipsoid. Let $\Gamma$ be a hyperplane parallel to $\Pi$ and with 
                     $O\in \Gamma$, let $\Pi_1,\Pi_2$ be support hyperplanes of $K$ at the points 
                     $x_1,x_2\subset \bd K$, respectively, parallel to $\Pi$ and let $L:=L(x_1, x_2)$. By 
                     Lemma \ref{teta}, there exists an affine transformation $A:\Rn \rightarrow \Rn$ such 
                    that $A(\Gamma)=\Gamma$, $A(L)\perp \Gamma$ and the ellipsoids 
                    $\{\Phi_u(A(K)): u\in (\mathbb{S}^{n-1}\cap \Gamma) \}$ are congruent ellipsoids of 
                    revolution with axis $A(L)$. It follows that $A(K)$ is an ellipsoid of revolution with axis 
                    $A(L)$.  Hence $K$ is an ellipsoid.
                    
                    \textbf{Acknowledment}. I want to thank the referee for his valuable observations and 
                    Su\-gges\-tions, which essentially allowed me to improve the manuscript, and to my dear 
                    friend to Jes\'us Jeronimo-Castro for His comments about this work.

            \textbf{Author Contributions.} Material preparation were performed by Efr\'en Morales 
           Amaya. Efr\'en Morales Amaya wrote the first draft of the manuscript.  
           
          \textbf{Funding.} The authors declare that no funds, grants, or other support were received 
          during the preparation of this manuscript.

          \textbf{Data Availability.} Data sharing not applicable to this article, as no datasets were 
          generated or analyzed during the current study.
          
          \textbf{Declarations.}
          
           \textbf{Conflict of interest.} The authors have no relevant financial or non-financial interests to disclose.

\end{document}